\documentclass[10pt]{amsart}
\usepackage{a4,amsmath,amssymb,amsthm,eucal}
\usepackage[english]{babel}
\usepackage{graphicx}
\usepackage{xcolor}
\usepackage{epsfig}

\newcommand\N{\mathbb{N}}
\newcommand\R{\mathbb{R}}

\newcommand{\forget}[1]{}
\newcommand{\HH}{{\mathcal H}}

\def\Om{{\Omega}}  %%%{{\bar{\Omega}}}
\def\om2{{\Om\times\Om}}

\def\supp{\mathrm{supp}\,}

\def\div{\mathrm{div}\,}

\def\Lip{\mathrm{Lip}}

\newcommand{\A}{\mathcal{A}}

\usepackage{color}

\newtheorem{theorem}{Theorem}[section]
\newtheorem{definition}[theorem]{Definition}

\newtheorem{lemma}[theorem]{Lemma}
\newtheorem{proposition}[theorem]{Proposition}
\newtheorem{corollary}[theorem]{Corollary}

\theoremstyle{remark}
\newtheorem{remark}[theorem]{Remark}
\newtheorem{example}[theorem]{Example}

\numberwithin{equation}{section}

\begin{document}
\title{The saga of a fish: from a survival guide to closing lemmas}
\author{Sergey Kryzhevich}
\address[Sergey Kryzhevich]{University of Nova Gorica, Vipavska,13, Nova Gorica, SI-5000, Slovenia
\and
Department of Mathematical Physics, Faculty of Mathematics and Mechanics,
St. Petersburg State University, Universitet\-skij pr.~28, Old Peterhof,
198504 St.Peters\-burg, Russia
}
\email[Sergey Kryzhevich]{kryzhevich@gmail.com}
\author{Eugene Stepanov}
\address[Eugene Stepanov]{
St.Petersburg Branch
of the Steklov Mathematical Institute of the Russian Academy of Sciences,
Fontanka 27,
191023 St.Petersburg,
Russia
\and
Higher School of Economics, Faculty of Mathematics, Usacheva str. 6, 119048 Moscow, Russia
%Department of Mathematical Physics, Faculty of Mathematics and Mechanics,
%St. Petersburg State University, Universitet\-skij pr.~28, Old Peterhof,
%198504 St.Peters\-burg, Russia
%\and ITMO University, Russia
}
\email[Eugene Stepanov]{stepanov.eugene@gmail.com}
\thanks{The first author was partially supported by RFBR grant \#18-01-00230-a. Both authors are grateful to Prof. S.~Crovisier for useful advices and discussion. The work of the second author
is %supported by the Program of the Presidium of the Russian
%Academy of Sciences \#01 ``Fundamental Mathematics and its Applications''
%under grant PRAS-18-01, and
partially supported by RFBR grant \#17-01-00678.
%and by the Russian
%government grant \#08-08, the Ministry of Education and Science of Russian Federation project
%\#14.Z50.31.0031.
}

\date{December 25, 2017}

\begin{abstract}
In the paper by D.~Burago, S.~Ivanov and A.~Novikov, ``A survival guide for feeble fish'', it has been shown that a fish with limited velocity can reach any point in the (possibly unbounded) ocean provided that the fluid velocity field is incompressible, bounded and has vanishing mean drift. This result extends some known global controllability theorems
though being substantially nonconstructive. We give a fish a different recipe of how to survive in a turbulent ocean, and show its relationship to structural stability of dynamical systems by providing a constructive way to change slightly the velocity field to produce conservative (in the sense of not having wandering sets of positive measure) dynamics. In particular, this leads to the extension of C.~Pugh's closing lemma to incompressible vector fields over unbounded domains. The results are based on an extension of the Poincar\'{e} recurrence theorem to some $\sigma$-finite measures and on specially constructed Newtonian potentials.
\end{abstract}


\maketitle

%\tableofcontents
%\newpage

\section{Introduction}

Consider the following problem first suggested in~\cite{BurNovIv16-fish}: a fish in an unbounded turbulent ocean is
able to move with its own velocity $u$ not exceeding in modulus the given value $\delta>0$. The ocean is assumed to be unbounded and is identified hereinafter with $\R^d$, its velocity field $V$ is assumed, as it is customary, to be bounded and \emph{incompressible}, i.e. $\mbox{div}\, V=0$. One is asked whether the fish can reach any point starting from an arbitrary one.
This is clearly the classical point-to-point controllability problem for the
autonomous system of ordinary differential equations
\begin{equation}\label{eq_odeaut1}
\dot x = V(x),
\end{equation}
where $x(\cdot)\in \R^d$. Namely, one has to
find a control function $u(\cdot)$ in some class of admissible controls such that
the trajectory $x(\cdot)$ of the system
\begin{equation}\label{eq_odectrl1}
\dot x = V(x) + u(t),
\end{equation}
starting at a given $x_0\in \R^d$ at time $t=0$ (i.e.\ satisfying
$x(0)=x_0$) arrives at some given $x_1\in \R^d$ at some finite time $\tau>0$, i.e.\ has $x(\tau)= x_1$.
The usual choice of the class of admissible controls is that of piecewise continuous functions $u\colon \R^+\to \R^d$ satisfying
$\|u\|_\infty \leq \delta$, where $\|\cdot\|_\infty$ is the supremum norm, and $\delta>0$ is given.

In case when the phase space of~\eqref{eq_odectrl1} instead of $\R^d$ is just a compact subset of the latter %(forward)
 invariant with respect to the flow of~\eqref{eq_odeaut1} (in particular, a smooth compact manifold), then the positive answer to the posed question is provided by the global controllability theorem~4.2.7 in~\cite{Bloch15-control} (there it is formulated for analytic vector fields
on compact Riemannian manifolds), while in the whole $\R^d$
the incompressibility condition of $V$ is clearly not enough for such a theorem to hold as can be seen just taking $V$ to be constant with sufficiently large modulus. This is in fact the only possible obstacle for controllability in the whole $\R^d$: in fact, it has been proven in~\cite{BurNovIv16-fish} that if $V$ is incompressible and has \emph{vanishing mean drift}
(called small mean drift in~\cite{BurNovIv16-fish}) in the sense
\[
\lim_{\ell\to \infty} \sup_{x\in \R^d}
\left|\frac{1}{\ell^d}\int_{[0,\ell]^d} V(x+y)\, dy\right|=0,
\]
then the above controllability problem in $\R^d$ is solvable for every couple of points $x_0$ and $x_1$ (this result has been further extended in~\cite{BurNovIv17-fish} to nonautonomous equations).
The respective proof is however by contradiction and hence strongly nonconstructive. In other words, one assures the fish that
it can reach any given destination without giving any clue on how to do that.

In search for a more constructive solution one might ask whether one can take the control to be given by a simple feedback
$u(\cdot)= W(x(\cdot))$, for some a priori unknown vector field $W\colon \R^d\to \R^d$, that could be explicitly constructed.
Once one looks at this problem under such a point of view, one immediately observes that in a particular case of
the \emph{return problem} (i.e.\ with $x_0=x_1$) the answer is positive when $x_0$ is a \emph{nonwandering} point of~\eqref{eq_odeaut1}, that is, in any neighborhood of $x_0$ there are points which return infinitely many times
in this neighborhood under the flow of~\eqref{eq_odeaut1}.
It is in fact provided by the famous Pugh's closing lemma~\cite{Pugh67}, which says that the perturbation $W$ of the original vector field $V$ can be
taken arbitrarily small not only in the uniform norm, but even in Lipschitz norm (sometimes in the dynamical systems literature  called $C^1$ norm).
The latter lemma is one of the fundamental results of structural stability theory for smooth dynamical systems, and is the first one of a series of similar results, the most well-known of which are the Ma\~{n}e ergodic closing lemma~\cite{Mane82}, the Hayashi connecting lemma~\cite{Hayashi97} and Bonatti and Crovisier connecting lemma~\cite{BonCrov04}
(see~\cite{AnosovZhuzh12} for the comprehensive overview on the subject).
It is important however to emphasize here the simple though quite striking observation that the
incompressibility of $V$, i.e.\ invariance of the Lebesgue measure under the flow induced by~\eqref{eq_odeaut1}, says nothing about existence of nonwandering points for this equation (as can be seen just by the example of a constant vector field $V$). %~\eqref{eq_odeaut1}
This is in sharp contrast with the case when~\eqref{eq_odeaut1} has a \emph{finite} invariant measure (which holds in particular when it has a compact invariant set): in the latter case all the points of the support of the invariant measure are nonwandering by the classical Poincar\'{e} recurrence theorem, and even assuming additionally the small mean drift condition for $V$ does not a priori improve this situation (see Remark~\ref{rm_Wdiv2}).

Our first principal result (Theorem~\ref{th_Pugh0}) shows that in fact under just incompressibility and vanishing mean drift condition on $V$ one can perturb the latter vector field by a small perturbation $W$ (even with small derivatives) so that \emph{every} point of $\R^d$ becomes
nonwandering with respect to the flow of $V+W$. This will be done by an explicit construction using Newtonian potential so as to ensure that the flow of $V+W$ preserve a new invariant measure, the support of which is the whole $\R^d$. This measure will still be not finite so that the classical Poincar\'{e} recurrence theorem cannot be applied, but will ``grow not too fast at infinity'', which will be shown to be enough for the extension of the latter theorem (Proposition~\ref{prop_Poincare1} and Corollary~\ref{co_Poincare1}) to hold. This
result combined with Pugh's closing lemma immediately implies its extension (Theorem~\ref{th_Pugh1}), namely, that
in fact \emph{every} chosen point of $\R^d$ can be made periodic
for a dynamical system provided by an ODE with incompressible vector field with vanishing mean drift at the right hand side
up to a small perturbation of the vector field in the Lipschitz norm.
Finally, we show that our construction actually implies the Burago-Ivanov-Novikov controllability theorem (Theorem~\ref{th_control1}), with a proof conceptually different from the original one, but very close to that of the classical controllability results for affine control systems with recurrent drift (see theorem~5 from~\cite[chapter~4]{Jurdj97} or theorem~4.2.7 in~\cite{Bloch15-control}). All the mentioned results are heavily based on the estimates for gradients
of constructed potentials, which are located in Appendix~\ref{sec_estcorr1} and use the results on vanishing mean drift condition from Appendix~\ref{sec_vmd0}.

\section{Notation and preliminaries}

%For a metric space $X$ endowed with distance $\rho$, a set $B\subset X$, and an $r>0$
%we denote by $(B)_r$ the $r$-neighborhood of $B$, namely,
%$(B)_r:=\{x\in X\colon \dist(x, B) <r\}$, by $B^c$ the complement
%of $B$ in $X$.

The finite-dimensional space $\R^d$ is assumed to be equipped
with the Euclidean norm $|\cdot|$,
notation $B_r(x)\subset \R^d$ stands for the usual open Euclidean ball of radius
$r$ centered at $x$, and $e_j$ to the $j$-th unit Cartesian coordinate vector.
The volume of the unit ball $B_1(0)\subset \R^d$ will be denoted $\omega_d$.
The usual scalar product of $x\in \R^d$ and $y\in \R^d$ is denoted by $x\cdot y$.
For any set $D\subset \R^d$, we let $\bar D$ be the closure of $D$,
$\mathbf{1}_D$ be its characteristic function,
$D^c:=\R^d\setminus D$.
%$\dist(x,D):=\inf\{|x-y|\,:\, y\in D\}$ whenever $x\in \R^d$.
%and $(D)_\varepsilon:=\{x\in E\,:\,\dist(x,D)< \varepsilon\}$.

We denote by $C(\R^d;\R^m)$ (respectively $C^k(\R^d;\R^m)$,
$C^{k,\beta}_{loc}(\R^d;\R^m)$,
$\Lip(\R^d;\R^m)$, $L^1(\R^d;\R^m)$, $L^\infty(\R^d;\R^m)$)
the set of continuous (respectively $k$-times continuously differentiable,
$k$-times continuously differentiable with
locally $\beta$-H\"{o}lder $k$-th derivatives, Lipschitz,
Lebesgue integrable, Lebesgue measurable and essentially bounded)
functions $f\colon \R^d\to \R^m$, omitting the reference for $\R^m$ when $m=1$, i.e. for real valued functions. The standard uniform norms in $C(\R^d;\R^m)$ and in $L^\infty(\R^d;\R^m)$ will be both denoted by $\|\cdot\|_\infty$.
By $L^1(\partial B_1(0);\HH^{d-1})$ we denote the class of functions
over $\partial B_1(0)$ integrable with respect to the $(d-1)$-dimensional Hausdorff measure $\HH^{d-1}$.
For a $V\in \Lip(\R^d;\R^m)$ we denote by $\Lip\, V$ its (least) Lipschitz constant, and also use the notation
$\|V\|_{\Lip}:=\|V\|_\infty +\Lip\, V$.
By $*$ we denote the convolution of functions.

All the measures
considered in the sequel are positive Radon measures, not necessarily finite.
For a Borel measure $\mu$ over a metric space $X$
we let $\supp\mu$ stand for its support, and for a Borel
map $T\colon X\to Y$ between metric spaces $X$ and $Y$ we denote
by $T_{\#}\mu$ the push-forward of $\mu$, i.e.\ the measure over
$Y$ defined by
$(T_{\#}\mu)(B):= \mu(T^{-1}(B))$ for every Borel $B\subset Y$.
For a metric space $X$
the set $M\subset X$ will be called %positively
invariant
for the map $T\colon X\to X$,
if $T(M)=M$, and the measure $\mu$ over $X$ will be called
%positively
invariant  for this map, if $T_{\#}\mu=\mu$
(of course, the map is assumed Borel in the latter case).
%For the
%sake of brevity we will further call positively invariant measures
%just invariant.
%Finally, $\HH^k$ will stand for the $k$-dimensional Hausdorff measure in $\R^d$.

\section{Accessibility}

In what follows we suppose that~\eqref{eq_odeaut1} is uniquely solvable and defined a flow
$T_t\colon \R^d\to \R^d$, $t\in \R$ %a flow of~\eqref{eq_odeaut1},
 by the formula $T_t(y):= x(t)$, where $x(\cdot)$ is a solution of~\eqref{eq_odeaut1}
satisfying $x(0)=y$. In the sequel, we will also denote the flow $T_t$ by $\varphi^t_V$ when we need to emphasize
that it is produced by the vector field $V$.
As the set of admissible controls
$U_\delta$ with given $\delta>0$ we consider, as usual in control theory,
the set of piecewise continuous functions $u\colon \R\to \R^d$ with
$\|u\|_\infty < \delta$, and set $U_0:=\{0\}$.
We recall the following definitions.

\begin{definition}\label{def_access1}
Given $\delta>0$, we say that a point $z\in \R^d$ is $\delta$-accessible from an $y\in \R^d$ in finite time
$\tau>0$, if there is an admissible control $u\in U_\delta$ such that
a trajectory of~\eqref{eq_odectrl1} with initial condition $x(0)=y$ arrives
in $z$ before time $\tau$, i.e.\ $x(s)=z$ for some $s\in [0,\tau]$. The set of
such points will be denoted by $\A(y,\tau,U_\delta)$.
We will also refer to
\[
\A(y,U_\delta):=\cup_{\tau>0} \A(y,\tau,U_\delta)
\]
as the set of points accessible from $y\in \R^d$ using controls in $U_\delta$.
%  and, for a set $M\subset \R^d$, we denote by
% \[
% \A(M,U_\delta):=\cup_{y\in M} \A(y,U_\delta),
% \]
% which is the set of points accessible from $M$ using controls in $U_\delta$.
\end{definition}

We also recall the following classical notions.

\begin{definition}\label{def_poisson1}
We say that
a set $\omega_x\subset\R^d$ is the $\omega$-limit set of $x\in \R^d$ for the flow $T_t$, if
it is the set of limit points of trajectories of~\eqref{eq_odeaut1}
starting at $x$, i.e.\ a set of $y\in \R^d$ such that there exist
a sequence $t_k\to +\infty$ (depending on y) such that
$T_{t_k}(x)\to y$ as $k\to \infty$.

A point $x\in \R^d$ is called
\begin{itemize}
\item[(i)] (forward) Poisson stable for the flow $T_t$, if $x\in \omega_x$, where
$\omega_x\subset\R^d$ is the $\omega$-limit set of $x$;
\item[(ii)] nonwandering for the flow $T_t$, if for every open $U\subset \R^d$, $x\in U$
one has $T_{t_k}(U)\cap U\neq \emptyset$ for some
sequence $t_k\to +\infty$ (depending on $U$).
\end{itemize}
\end{definition}

Clearly the closure of the set of Poisson stable points for the flow $T_t$ is contained in the set of nonwandering points for the latter (called usually \emph{nonwandering set}).

We will further use the following technical and probably folkloric (though not easily found in the literature) result (in a sense, a vaguely similar assertion is theorem~5 from~\cite[chapter~4]{Jurdj97}).

\begin{proposition}\label{prop_trans1}
Let $V\colon \R^d\to \R^d$ be bounded and %uniformly
continuous, such that~\eqref{eq_odeaut1} is uniquely solvable, hence
generating the flow $T_t:= \varphi^t_V$. Let also
$M\subset \R^d$
be a closed set %positively
invariant with respect to $T_t$
%the flow $T_t:= \varphi^t_V$
%of~\eqref{eq_odeaut1}
such that
the set of Poisson stable points of $T_t$ is dense in $M$.
% Assume, moreover, that $T_t$ sends closed subsets of $M$ into closed ones
%  (in particular, this holds when $T_t$ is a homeomorphism, or when $T_t$
% is just continuous and $M$ is compact).
If
%$M_0\subset M$ is a closed set %positively
%invariant with respect to $T_t$
%that
$M$ cannot be represented as a disjoint union of two non-empty closed
 subsets %positively
 invariant with respect to $T_t$, then
%$M_0\subset \A(x_0,U_\delta)$
$M\subset \A(x_0,U_\delta)$ for every $\delta>0$ and every $x_0\in M$.
%$x_0\in M_0$.
\end{proposition}

\begin{proof}
Fix an arbitrary $x_0\in M$ and $\delta>0$.
Consider the set
\[
M_{x_0,\delta}:= M\cap \A(x_0, U_\delta)
%M\cap \bigcup_{\sigma <\delta} \A(x_0, U_\sigma)
\]
of all points $y\in M$
accessible from  $x_0$ using controls in
$U_\delta$.
% $U_\sigma$ with
% any $0\leq \sigma<\delta$.
Clearly, $M_{x_0,\delta}\neq\emptyset$
because $\A(x_0, U_0)\subset M_{x_0,\delta}$.
We now prove several consecutive claims.

{\sc Step 1}. We show that $M_{x_0,\delta}$
is relatively open in $M$. In fact, $z\in M_{x_0,\delta}$
%for a
%$\sigma<\delta$
means that
$z=x(s)$ for an $s>0$ and some trajectory $x(\cdot)$ of~\eqref{eq_odectrl1}
with $x(0)=x_0$ and
%$\|u\|_\infty \leq \sigma$.
 $\|u\|_\infty < \delta$.

%%%%%%%%%%%%%%%%%%%%%%%%%%%%%% VERY NEW 20.01.19 %%%%%%%%%%%%%%%%%%%%%%%%%%%

Choose now an $\varepsilon>0$ so small that
\begin{equation}\label{eq_corr1}
|V(y)-V(z)|< \frac{\delta - \|u\|_\infty}{4},
\end{equation}
for all $y\in B_\varepsilon(z)$, and a $\tau\in (0,s)$ (depending on $\varepsilon$) so small that
\begin{eqnarray}
\label{eq_corr3}
(\|V\|_\infty +\delta)\tau< \varepsilon/2,\\
\label{eq_corr4}
|x(s-\tau)-z|<\varepsilon/2.
\end{eqnarray}
Letting $\bar x(\cdot)$ over $[s-\tau, s]$ stand for the trajectory
of the ODE $\dot{\bar x}(t)= V(z)+u(t)$ satisfying $\bar x (s-\tau)=x(s-\tau)$,
we get that both
$\bar x(t)\in B_\varepsilon(z)$ and $x(t)\in B_\varepsilon(z)$ for all
$t\in [s-\tau, s]$ due to~\eqref{eq_corr3}  and~\eqref{eq_corr4}.
Therefore, from~\eqref{eq_corr1} we get
\begin{equation}\label{eq_corr1a}
|z-\bar x(s)|=|x(s)-\bar x(s)| < \tau \frac{\delta - \|u\|_\infty}{4}.
\end{equation}

For an $\alpha\in \R^d$ denote by $x_\alpha(\cdot)$ over $[s-\tau, s]$ the trajectory
of the ODE
\[
\dot{x}_\alpha(t)= V(z)+u(t)+\alpha
\]
 satisfying
$x_\alpha(s-\tau)= x(s-\tau)$. We have then
that
\[
x_\alpha(t) = \bar x(t) +\alpha t,
\]
and in particular,
\[
\left \{x_\alpha(s)\colon |\alpha|<  \frac{\delta - \|u\|_\infty}{2}\right\}
= B_r (\bar x(s)),\quad \mbox{where}\quad
r:= \tau \frac{\delta - \|u\|_\infty}{2},
\]
which contains $z$ and hence an open ball centered in $z$ in view of~\eqref{eq_corr1a}.
It is enough then to set
\[
u_\alpha(t):=
\left\{
\begin{array}{rl}
u(t), & t\in [0, s-\tau],\\
V(z)- V(x_\alpha(t)) +u(t) +\alpha, & t\in (s-\tau, s],
\end{array}
\right.
\]
to get
that $x_\alpha$ is a trajectory of~\eqref{eq_odectrl1} over $[0,s]$
(with $u_\alpha$ instead of $u$).
Since~\eqref{eq_corr3} implies that also $x_\alpha(t)\in B_\varepsilon(z)$
for all $t\in [s-\tau, s]$, then from the definition of $u_\alpha$ we obtain
for $|\alpha|< (\delta-\|u\|_\infty)/2$ the estimate
\begin{align*}
|u_\alpha(t)|& \leq |u(t)|+ |\alpha|+ |V(z)- V(x_\alpha(t))|\\
& \leq
\|u\|_\infty + \frac{\delta - \|u\|_\infty}{2} + \sup_{y\in B_\varepsilon(z)} |V(y)-V(z)|\\
& \leq \|u\|_\infty + \frac{\delta - \|u\|_\infty}{2} +
\frac{\delta - \|u\|_\infty}{4} \quad\mbox{by~\eqref{eq_corr1}}
\\
& < \delta,
\end{align*}
which proves the claim.

{\sc Step 2}. We show
\begin{equation}\label{eq_Mdinv1}
\overline{T_s(M_{x_0,\delta})}=M_{x_0,\delta}\quad \mbox{for all $s\in \R^+$}.
\end{equation}
To show
$\overline{T_s(M_{x_0,\delta})}\subset M_{x_0,\delta}$,
observe that
$T_s(M_{x_0,\delta})\subset M_{x_0,\delta}$, and therefore, as proven on Step~1,
 $M_{x_0,\delta}$ contains a relatively open neighborhood of $T_s(M_{x_0,\delta})$, which implies
the desired inclusion for the closure.

To prove the converse inclusion
$M_{x_0,\delta}\subset \overline{T_s(M_{x_0,\delta})}$,
we note that the set of Poisson stable points belonging to $M_{x_0,\delta}$
is dense in  $M_{x_0,\delta}$, because the latter set is relatiely open in $M$ by
Step~1.
Take now an arbitrary Poisson stable point
$p\in M_{x_0,\delta}$,
so that
$p=\lim_k x(t_k)$ where $x(\cdot)$ is a solution of~\eqref{eq_odeaut1} with
$x(0)=p$, %where $p_k\in T_{t_k}(p)$
for some sequence $t_k \to +\infty$.
Clearly $x(t)\in M_{x_0,\delta}$ for every $t\in\R^+$, and hence
$q_k:=x(t_k-s)\in M_{x_0,\delta}$. By construction
$x(t_k)\in T_s(q_k)\subset T_s(M_{x_0,\delta})$, and hence
$p\in \overline{T_s(M_{x_0,\delta})}$. The inclusion being proven
follows then from density of Poisson stable points in $M_{x_0,\delta}$.

{\sc Step 3}. As a result of Step~2 we have that
$M_{x_0,\delta}$ is closed.
Therefore,
%$M_0\subset M_{x_0,\delta}$,
$M\subset M_{x_0,\delta}$
since otherwise disjoint sets
%$M_0\bigcap M_{x_0,\delta}$ and
%$M_0\setminus M_{x_0,\delta}$
$M\bigcap M_{x_0,\delta}$ and
$M\setminus M_{x_0,\delta}$
are both closed, nonempty and %positively
invariant for $T_t$
for all $t\in \R^+$ contrary to the assumption.
\end{proof}

\begin{remark}\label{rem_trans1}
Note that the above Proposition~\ref{prop_trans1} does not require $V$ to be locally Lipschitz, nor even uniformly continuous.
% However, it remains true if $V$ is assumed locally Lipschitz
% and not necessarily uniformly continuous; in fact,
% in Step~1 of its proof it is enough to have $\tilde V$ which is
% locally Lipschitz over $\R^d$ and continuously differentiable over
% $B_{2R}(x_0)$. To construct such an approximation of $V$ we first construct a
% continuously differentiable approximation $V_1$ of the latter by convolution
% with a smooth approximate identity with  $\|(V-V_1)\mathbf{1}_{B_{2R}(x_0)}\|_\infty < \varepsilon$, $\|V_1\|_\infty\leq \|V\|_\infty$, and then
% set $\tilde V:= V_1 \eta_R + V (1-\eta_R)$ where $\eta_R\colon \R^d\to [0,1]$ is
% a smooth function which iz identically one over $B_R(x_0)$ and vanish outside
% of $B_{2R}(x_0)$.
\end{remark}

\section{Controllability and closing lemma}

\subsection{Poincar\'{e} recurrence theorem for not necessarily finite measures}

We need the following generalization of the Poincar\'{e} recurrence theorem, which we formulate for injective maps, since this will be more adapted to the use in the sequel.

\begin{proposition}\label{prop_Poincare1}
Let $T\colon\R^d \to \R^d$ be an injective Borel map preserving a $\sigma$-finite Borel measure $\mu$.
Suppose also that
\[
\mu\left(\bigcup_{k=0}^n T^k(B_\rho(0))\right)=o(n) \quad \mbox{as $n\to +\infty$}
\]
for every $\rho>0$ fixed.
Then for every open $U\subset \R^d$ such that $\mu(U)>0$ one has
$T^n(U) \bigcap U \neq \emptyset$ for a subsequence of $n\in \N$ (depending on $U$).
\end{proposition}

\begin{remark}\label{rem_Poincare1}
The conditions of Proposition~\ref{prop_Poincare1} are in particular satisfied
in the following cases which are of practical importance:
\begin{itemize}
\item[(i)] $\mu$ is finite (this recovers the classical statement of the Poincar\'{e} recurrence theorem;
\item[(ii)]
%for all $n\in \N$ and $\rho\geq 0$
%and
for
some function $\kappa\colon \R^+\to \R^+$
one has
\begin{align*}
\mu(B_\rho(0)) = o(\kappa(\rho))&\quad \mbox{as $\rho\to +\infty$},\\
T^n(B_\rho(0))\subset B_R(0) & \quad \mbox{with
$R=R(\rho,n)>0$, $R(\rho,\cdot)$ nondecreasing}\\
& \quad \mbox{and
$\kappa(R)=O(n)$ as $n\to +\infty$ for every $\rho>0$ fixed};
\end{align*}
\item[(iii)]
$|T^n(x)|\leq A|x|+Bn$ for some $A\geq 0$, $B\geq 0$, and
$\mu(B_\rho(0))= o(\rho)$ as $\rho\to +\infty$, as can be seen from
taking $\kappa$ linear in~(ii).
\end{itemize}
\end{remark}

\begin{proof}
Since it is enough to prove the statement for each $U\cap B_\rho(0)$, for an
arbitrary $\rho>0$, we may assume without loss of generality that $U\subset B_\rho(0)$
for some $\rho>0$.
Fix an arbitrary $n\in \N$ and
consider the set
\[
U_n:=
\{x\in U\colon T^k(x) \notin U \quad \mbox{for all $k\geq n$}\}.
\]
Let $F:= T^n$ and since $F^k$ are injectve and preserve $\mu$, we get
\[
\mu(F^k(U))=\mu(F^{-k}(F^k(U)))=\mu(U)
\]
for all $k\in \N$.
Clearly, $\{F^k(U_n)\}_{k\in \N}$ are pairwise disjoint, since otherwise one would have
$D:=F^m(U_n)\cap F^i(U_n)\neq \emptyset$ for some $m>i\geq n$, and hence, using again injectivity of $F^k$, we would get
\[
(F^i)^{-1}(D)= F^{m-i}(U_n)\cap U_n =T^{(m-i)n}(U_n)\cap U_n\neq
\emptyset
\]
contradicting the definition of $U_n$.
Then
\begin{align*}
 m\mu(U_n)&=\mu\left(\bigcup_{k=0}^{m-1} F^k(U_n)\right)\quad
\mbox{because $F^k(U_n)$ are pairwise disjoint}\\
& \leq \mu\left(\bigcup_{k=0}^{(m-1)n} T^k(U_n)\right) =o(m)
 \end{align*}
as $m\to \infty$,
which implies $\mu(U_n)=0$.
\end{proof}

% \begin{proof}
% Since it is enough to prove the statement for each $U\cap B_\rho(0)$, for an
% arbitrary $\rho>0$, we may assume without loss of generality that $U\subset B_\rho(0)$
% for some $\rho>0$.
% First we show that $\mu(U\cap T^{-m}(U))>0$ for some $m>0$, if $\mu(U)>0$.
% In fact,
% $\mu(T^{-m}(U))= \mu(U)$ for all $m\in \N$,
% and hence %there is a $k\in N$ such that
% the sets
% $T^{-j}(U)$ cannot be all pairwise disjoint,
% %for every $j\geq k$,
% because otherwise %(if all $T^{-j}(U)$ are pairwise disjoint),
% we would have
% \begin{align*}
%  m\mu(U)&=\mu\left(\bigcup_{j=0}^{m-1} T^{-j}(U)\right) = o(m)
%  \end{align*}
% as $m\to \infty$, which is a contradiction when $\mu(U)>0$.
% Therefore there are two different indices, $i$ and $j$, say, $j>i$ for
% definiteness, such that
% \[
% 0< \mu(T^{-j}(U)\cap T^{-i}(U))
% \]

% Fix an arbitrary $n\in \N$ and
% consider the set
% \[
% U_n:=
% \{x\in U\colon T^k(x) \notin U \quad \mbox{for all $k\geq n$}\}.
% \]
% Let $F:= T^n$ and recall that $\mu(F^k(U))=\mu(U)$ for all $k\in \N$.
% Then
% \begin{align*}
%  m\mu(U_n)&=\mu\left(\bigcup_{k=0}^{m-1} F^k(U_n)\right)\quad
% \mbox{because $F^k(U_n)$ are pairwise disjoint}\\
% & \leq \mu\left(\bigcup_{k=0}^{(m-1)n} T^k(U_n)\right) =o(m)
%  \end{align*}
% as $m\to \infty$,
% which implies $\mu(U_n)=0$.
% \end{proof}

The following corollary gives a ``continuous'' version of this statement
adapted to our setting.

\begin{corollary}\label{co_Poincare1}
Let $V\colon\R^d \to \R^d$ be a bounded locally
Lipschitz vector field, and
$\mu$ be a $\sigma$-finite Borel measure invariant with respect to the flow $\varphi_V^t$ induced by $V$ and
satisfying $\mu(B_R(0))= o(R)$ as $R\to +\infty$.
Then for every open
$U\subset \R^d$ and for every $\tau>0$ one has
\[
\mu\left(
\left\{
x\in U\colon \varphi_V^t (x)\not\in U\mbox{ for all }t\geq \tau
\right\}
\right)=0.
\]
\end{corollary}

\begin{proof}
Fix an arbitrary $\tau>0$ and let $T:=\varphi_V^\tau$.
The statement follows from Proposition~\ref{prop_Poincare1} together with Remark~\ref{rem_Poincare1}(iii) (with $A:=1$, $B:= \|V\|_\infty$, since
$\|\varphi_V^t(x)\|\leq |x|+ \|V\|_\infty t$ for $t\geq 0$).
\end{proof}

\begin{remark}\label{rem_Poincare2}
The above Corollary~\ref{co_Poincare1}
remains valid for locally Lipschitz but  not necessarily bounded
vector fileds $V$ satisfying one-sided Osgood estimate
\[
\sup_{|x|\le r} \frac{V(x)\cdot x}{|x|}\leq K(r)
\]
for all $r>0$
and for some %continuous
measurable
function $K\colon \R^+\to \R^+$
such that
\[
\int_0^\infty \frac{\,ds}{K(s)} =+\infty,
\]
under the condition
\[
\mu(B_\rho(0))=o\left(\int_0^\rho \frac{\,ds}{K(s)}\right),
\]
as $\rho\to \infty$.
In fact, denoting
\[
\kappa(\rho):=\int_0^\rho \frac{\,ds}{K(s)},
\]
we get that $\kappa$ is strictly increasing with $\kappa(t)\to +\infty$ as $t\to +\infty$. Furthermore,
\begin{equation}\label{eq_estkappa1}
|\varphi_V^t(x)|\le R(|x|,t):=\kappa^{-1}(t+\kappa(|x|)),
\end{equation}
because
for $r(t):=|x(t)|$, $x(t):=\varphi_V^t(x)$ one has
\[
\dot{r}(t)=\frac{V(x(t))\cdot x(t)}{r(t)}\leq K(r(t)),
\]
and integrating the latter inequality one gets
$\kappa(r(t))-\kappa(r(0))\le t$, which implies~\eqref{eq_estkappa1}.
Thus, $\varphi_V^t(B_\rho(0)) \subset B_{R(\rho,t)}(0)$ with
$\kappa(R(\rho,t))=O(t)$ as $t\to+\infty$, and
$\mu(B_\rho(0))=o\left(\kappa(\rho)\right)$ as $\rho\to +\infty$, and hence
in the proof of Corollary~\ref{co_Poincare1} it is enough to
refer to Remark~\ref{rem_Poincare1}(ii) instead of
Remark~\ref{rem_Poincare1}(iii).
\end{remark}

We will further use also the following statement.

\begin{corollary}\label{co_Poincare3}
Under conditions of Corollary~\ref{co_Poincare1} (or, more generally, of the Remark~\ref{rem_Poincare2}) one has that Poisson stable points of $V$ are dense
in $\supp\,\mu$, and in particular, all the points of  $\supp\,\mu$ are nonwandering for the flow generated by $V$.
\end{corollary}

\begin{proof}
We repeat the main idea of the part~1 of proposition~4.1.18 from~\cite{KatHass95}.
Let $\{U_j\}_{j\in \N}$ be a countable base of the topology in $\R^d$.
For every $j\in \N$ we consider
\begin{align*}
N_j &:=
\{ x\in U_j:  \varphi_V^t(x) \notin U_j \, \mbox{for all $t \ge T_0$ and for some
$T_0=T_0(x)\in \R$}
\}, \\
R&:= \bigcap_j N_j^c.
\end{align*}
In other words, $N_j$ is the set of all points of $U_j$, the
iterations of which leave this set forever, while for every $x\in R\cap U_j$ there is a  sequence $t_k\to \infty$ such that $\varphi^{t_k}_V(x) \in U_j$.
By Corollary~\ref{co_Poincare1}, $\mu(N_j)=0$ for all $j\in \N$, and thus
$\mu(R^c)=0$, which implies density of $R$
in $\supp\, \mu$.

But if $x\in R$, then for any neighborhood $U$ of $x$ there is a
$U_j\subset U$ such that $x\in U_j$, and hence
for
a sequence $t_k\to +\infty$ one has $\varphi_V^{t_k}(x)\in U_j\subset U$
for sufficiently large $k\in \N$.
This means that the $\omega$-limit set $\omega_x$ of the point $x$ intersects with $U$. Since the set $\omega_x$ is closed and the neighborhood $U$ is arbitrary, we have $x\in \omega_x$ concluding the proof.
\end{proof}

\subsection{Correcting the vectorfield}

The basic idea of our construction is as follows.
Supposing that the Lebesgue measure over $\R^d$ is invariant under the
flow of the vector field $V$, we will ``correct'' the latter by
adding a new vector filed $W$ (referred later as \emph{corrector} such that the sum $V+W$ satisfy Corollary~\ref{co_Poincare1}  (or, more generally, of the Remark~\ref{rem_Poincare2}), and hence
also Corollary~\ref{co_Poincare3} for some new $\sigma$-finite measure $\mu$
with $\supp\,\mu=\R^d$ which will also be explcitly constructed. This will immediately lead to the proof of Burago-Ivanov-Novikov controllability theorem
once one shows that $\|W\|_\infty$ may be made arbitrarily small. We will further
show that this construction is in fact deeper and provides, for instance, a version of the Pugh closing lemma.

To fulfill this program define a positive function
\begin{equation}\label{eq_defmu1}
\psi(x):= (|x|^2+\alpha^2)^{-p},
\end{equation}
where $\alpha$ and $p$ are positive parameters to be defined later.
We will define a smooth map $W\colon \R^d\to \R^d$
depending on $\alpha$ and $p$ will so that
the measure $\mu:=\psi\,dx$ be invariant under the flow
of the perturbed system of ODEs
\begin{equation}\label{eq_ODEpert1}
\dot x=V(x)+W(x).
\end{equation}
% We will further show that with the appropriate choice of the parameters $t$, $p$ and $q$ one can make
% simultaneously $W$ arbitrarily small in supremum (and even in $C^1$) norm, and the measure $\mu$ to satisfy
% conditions of the generalized Poincar\'{e} lemma (Proposition~\ref{prop_Poincare1}). This will lead both to the controllability result and to the version of the Pugh closing lemma.
Making $\mu$ invariant with respect to the flow of~\eqref{eq_ODEpert1} amounts to making
\begin{equation}\label{eq_defWinv1a}
\mbox{div}\, \psi(V+W)=0
\end{equation}
or equivalently, recalling that $V$ is incompressible,
\begin{equation}\label{eq_defW1}
\mbox{div}\, (\psi W)= -\nabla \psi\cdot V
\end{equation}
in the weak sense. Letting $u$ stand for a solution to the Poisson equation
\[
%\begin{equation}\label{eq_Poisson1}
-\Delta u= \nabla \psi\cdot V
%\end{equation}
\]
in $\R^d$, we get that~\eqref{eq_defW1} is satisfied with
\begin{equation}\label{eq_defW2}
W:= \frac{1}{\psi}\nabla u.
\end{equation}
We may thus take $u$ of the form
\[
u:= \Phi*(\nabla \psi\cdot V),
\]
where $\Phi$ stands for the fundamental solution of the Laplace equation in $\R^d$, so that~\eqref{eq_defW2} reduces then to
\begin{equation}\label{eq_defW3}
\begin{aligned}
W(x) &:= -\frac{c_d}{\psi(x)} \int_{\R^d} \frac{x-y}{|x-y|^d} \nabla \psi(y)\cdot V(y)\, dy\\
 & = 2p c_d (|x|^2+\alpha^2)^p\int_{\R^d} \frac{x-y}{|x-y|^d} \frac{ y\cdot V(y)}{(|y|^2+\alpha^2)^{p+1}}
 \, dy,
\end{aligned}
\end{equation}
where $c_d:= 1/d\omega_d$.
Clearly, therefore, the following statement is valid.

\begin{lemma}\label{lm_Wmuinv1}
If $V\colon\R^d\to \R^d$ is locally Lipschitz, then
for every $p$ and $\alpha$ the vector field $W$ defined by~\eqref{eq_defW3}
is $C^{1,\beta}_{loc}$ smooth for every $\beta\in [0,1)$. If, moreover, $V$ is also bounded, incompressible and has vanishing mean drift, then
$W$ is bounded and
the measure
$\mu:=\psi\,dx$ is invariant under the flow of~\eqref{eq_ODEpert1}.
\end{lemma}

\begin{proof}
The smoothness of $W$ follows immediately from local elliptic Sobolev regularity together
with the Sobolev embedding theorem.
If $V$ is also bounded, incompressible and has vanishing mean drift, then boundedness of $W$ follows from Lemma~\ref{lm_WCsmall}.
Hence in this case the vector field $V+W$ generates the flow, and
the invariance of $\mu$ follows from~\eqref{eq_defWinv1a} in view of Lemma~\ref{lm_invmeas_div1}.
\end{proof}

We observe now that with the appropriate choice of the parameters the vector field $W$ can be made arbitrarily small
in supremum and, under a bit more requirements on regularity of $V$, even Lipschitz norm; this is the assertion of the following lemma which collects
several calculations made in the Appendices~\ref{sec_estcorr1} and~\ref{sec_vmd0}.

\begin{lemma}\label{lm_Wgensmall}
Let $p\in ((d-1)/2, d/2)$. Suppose that $V\in \Lip_{loc}(\R^d;\R^d)$
%$V$
is a bounded incompressible
vector field with vanishing mean drift.
Then, given an $\varepsilon>0$, there is an $\bar\alpha=\bar\alpha(p,\varepsilon)$ such that for every $\alpha> \bar\alpha$
the vector field $W$ defined by~\eqref{eq_defW3}
satisfies
\begin{itemize}
\item[(i)] $\|W\|_\infty \leq \varepsilon$,
\item[(ii)] $\|\mathrm{div}\, W \|_\infty  \leq \varepsilon$.
% \begin{align*}
% \|W\|_\infty &\leq \varepsilon,\\
% \|\mathrm{div}\, W \|_\infty & \leq \varepsilon\quad \mbox{
% for every $\alpha> \bar\alpha$}.
%\end{align*}
\end{itemize}
If, moreover, $V\in
\Lip(\R^d;\R^d)$ is incompressible vector field with vanishing mean drift, and
all $V_{x_j}$, $j=1, \ldots, d$, are locally Lipschitz still having vanishing mean drift, then one can choose $\bar\alpha$ so as to have additionally
\begin{itemize}
\item[(iii)] $\|W_{x_j}\|_\infty \leq \varepsilon$ for every $\alpha> \bar\alpha$
and for every $j=1,\ldots, d$.
% \[
% \|W_{x_j}\|_\infty \leq \varepsilon\quad \mbox{for every $\alpha> \bar\alpha$}
% \]
% for every $j=1,\ldots, d$.
\end{itemize}
In particular, the latter assertion holds when $V\in C^1(\R^d;\R^d)\cap \Lip(\R^d;\R^d)$
and has uniformly continuous first derivatives.
\end{lemma}

\begin{proof}
Assertions~(i) and~(ii) are Lemma~\ref{lm_WCsmall} and Corollary~\ref{co_Wdivsmall} respectively,~(iii) is Lemma~\ref{lm_WC1small}.
Finally, when $V\in C^1(\R^d;\R^d)\cap \Lip(\R^d;\R^d)$
and has uniformly continuous first derivatives, then
the first derivatives of $V$ also are incompressible and have vanishing mean drift by Lemma~\ref{lm_md_der1}, and so~(iii) still holds.
\end{proof}

\subsection{The end of the game: results}

The following statement is the first principal result of the paper.

\begin{theorem}\label{th_Pugh0}
Let $V\in \Lip_{loc}(\R^d;\R^d)$ be a bounded
incompressible vector field with %uniformly continuous first derivatives
%and
vanishing mean drift.
Then for every $\varepsilon>0$ there is a $C^{1,\beta}_{loc}$ (for every $\beta>0$) vector field
$W_\varepsilon\colon\R^d \to \R^d$
 with $\|W_\varepsilon\|_\infty \leq \varepsilon$ such that every $x\in \R^d$ is a nonwandering
 point of the
system of ODEs
\begin{equation}\label{eq_odePugh20}
\dot x = V(x)+W_\varepsilon(x).
\end{equation}
%and $\R^d$ is %forward
%invariant with respect to the flow defined by the latter.
If, moreover,
$V\in C^1(\R^d;\R^d)\cap \Lip(\R^d;\R^d)$
and has uniformly continuous first derivatives, then one may find $W_\varepsilon$ as above satisfying even the stronger estimate
$\|W_\varepsilon\|_{\Lip} \leq \varepsilon$.
\end{theorem}

\begin{proof}
We choose a $p\in ((d-1)/2,d/2)$
and
an $\alpha>0$ so as to have
$\|W\|_\infty \leq \varepsilon$
(resp.\ $\|W\|_{\Lip}\leq \varepsilon$ under the stronger regularity condition $V\in C^1(\R^d;\R^d)\cap \Lip(\R^d;\R^d)$ with uniformly continuous first derivatives), where
$W$ is defined by~\eqref{eq_defW3}: this is possible in view of
assertion~(i) (resp.~(iii)) of Lemma~\ref{lm_Wgensmall}.
%(when applying the Lemma~\ref{lm_Wgensmall}(iii) recall that the first derivatives of $V$ also are incompressible and have vanishing mean drift by Lemma~\ref{lm_md_der1}).
% Lemma~\ref{lm_WCsmall} (resp.\ Lemma~\ref{lm_WCsmall} together with Lemma~\ref{lm_WC1small}; when applying the latter lemma recall that the first derivatives of $V$ also are incompressible and have vanishing mean drift by Lemma~\ref{lm_md_der1}).
Since we have chosen $p> (d-1)/2$, then
\[
\mu(B_R(0)) =\int_{B_R(0)} \psi(x)\, dx = d\omega_d\int_0^R \frac{r^{d-1}\,dr}{(r^2+\alpha^2)^p}=o(R)
\]
as $R\to +\infty$,
and hence by Corollary~\ref{co_Poincare3} applied to $V+W$ (in place of $V$)
all the points of $\R^d=\supp\mu$ (note that $\supp\mu$ is %forward
invariant for the latter flow by Lemma~\ref{lm_invsupp1}) are nonwandering for the flow generated by the
vector field $V+W$.
It suffices to take then $W_\varepsilon:=W$.
\end{proof}

\begin{remark}\label{rm_Wdiv2}
It is important to emphasize that in contrast with the case when~\eqref{eq_odeaut1} has a \emph{finite} invariant measure
(e.g. when it has a compact invariant set),
an incompressible smooth vector field $V$ with vanishing mean drift
may produce a strongly dissipative dynamics in the sense of having a wandering set of full measure, as, for instance, when $d=2$ and, say, $V(x_1,x_2):= (0, \sin x_1)$.
%, so that a.e.
%point of $\R^2$ is a wandering point for the respective dynamical system.
In view of this observation the Theorem~\ref{th_Pugh0} is quite striking: it says that one may change the dynamics from strongly dissipative to a conservative one (i.e.\ with no wandering set)
over the whole (unbounded) space by an arbitrarily small  perturbation (even with small first derivatives) of the vector field.
\end{remark}

\begin{remark}\label{rm_Wdiv1}
It is worth observing that the perturbation $W_\varepsilon$, and hence the perturbed vector field $V+W_\varepsilon$
constructed in the proof
of the above Theorem~\ref{th_Pugh0} are in general not incompressible. However in view of Lemma~\ref{lm_Wgensmall}(ii) %Corollary~\ref{co_Wdivsmall}
one can ensure that the divergence of $W_\varepsilon$ (hence also that of $V+W_\varepsilon$) be
arbitrarily small in the uniform norm.
\end{remark}

The first corollary of the above theorem is the following global point-to-point controllability result, which is
a reformulation of the Burago-Ivanov-Novikov controllability theorem~1.1 from~\cite{BurNovIv16-fish}
(and a partial extension of theorem 4.2.7 in~\cite{Bloch15-control} formulated for compact manifolds, although for possibly more general control affine systems), however proven now by a completely different and direct method.

\begin{theorem}\label{th_control1}
Let $V\in \Lip_{loc}(\R^d;\R^d)$ be a bounded incompressible vector field with vanishing mean drift. Then for every couple of points $\{x_0,y_0\}\subset\R^d$ and every $\varepsilon>0$ there is a piecewise continuous function $u\colon \R^+\to\R^d$ (``control'') with $\|u\|_\infty\leq \varepsilon$ such that
the trajectory of the system of ODEs
\begin{equation}\label{eq_odectrl2}
\dot x = V(x)+u(t),
\end{equation}
satisfying $x(0)=x_0$ passes through $y_0$, i.e.\ $x(T)=y_0$ for some $T>0$.
\end{theorem}

\begin{proof}
Fixed an $\varepsilon>0$, by
Theorem~\ref{th_Pugh0} we find a smooth vector field
$W_\varepsilon\colon\R^d \to \R^d$
 with $\|W_\varepsilon\|_\infty \leq \varepsilon/2$ such that every $x\in \R^d$ is nonwandering
 with respect to the flow defined by $V+W_\varepsilon$ and $\R^d$ is %forward
 invariant with respect to the latter flow.
%We choose a $p\in ((d-1)/2,d/2)$
%and
%an $\alpha>0$ so as to have $\|W\|_{\Lip}\leq \varepsilon/2$, where
%$W$ is defined by~\eqref{eq_defW3} (this is possible in view of
% Lemma~\ref{lm_WCsmall}). Since we have chosen $p> (d-1)/2$, then
%\[
%\mu(B_R(0)) =\int_{B_R(0)} \psi(x)\, dx = d\omega_d\int_0^R \frac{r^{d-1}\,dr}{(r^2+\alpha^2)^p}=o(R),
%\]
%and hence by Corollary~\ref{co_Poincare3} aplied to $V+W$ (in place of $V$)
%all the points of $\R^d=\supp\mu$ are nonwandering for the flow generated by the
%vector field $V+W$. Note that $\supp\mu$ is invariant for the latter flow by Lemma~\ref{lm_invsupp1}.
Proposition~\ref{prop_trans1} (minding Remark~\ref{rem_trans1}) applied with $M:=\R^d$
and $V+W_\varepsilon$ instead of $V$ (in particular, $T_t$ standing for the flow generated by $V+W_\varepsilon$)
implies now the existence
of a piecewise continuous control $\tilde u\colon \R^+\to \R^d$ with
$\|\tilde u\|_\infty\leq \varepsilon/2$ such that the trajectory $x(\cdot)$ of the system
\[
\dot{x}= V(x)+W_\varepsilon(x)+\tilde u(t)
\]
starting at $x_0\in \R^d$ arrives  at $y_0\in \R^d$ in finite time (in alternative to
Proposition~\ref{prop_trans1} %and Remark~\ref{rem_trans1}
one could have used here theorem~5 from~\cite[chapter~4]{Jurdj97}).
It suffices to take now $u(t):= W_\varepsilon(x(t))+\tilde u(t)$.
\end{proof}

\begin{remark}\label{rm_Wdiv3}
If one extends a bit Proposition~\ref{prop_trans1} showing that one can achieve any given point in a compact set from a another point in the same set in finite time depending on the compact set, then under global Lipschitz continuity of $V$ one would have
also the estimate on arrival time to the destination as in theorem~1.2 of~\cite{BurNovIv16-fish} (and with exactly the same proof), though this is beyond the scope of the present paper.
\end{remark}

The following easy corollary slightly extends Theorem~\ref{th_control1}
to velocity fields which are just uniformly continuous (hence possibly even
not provide unique solvablity of~\eqref{eq_odeaut1}).

\begin{corollary}\label{co_control2}
If $V\colon\R^d\to \R^d$ is a bounded uniformly continuous
(not necessarily locally Lipschitz)
incompressible (in the weak sense) vector field with vanishing mean drift, then
for every couple of points $\{x_0,y_0\}\subset\R^d$ and every $\varepsilon>0$ there is a piecewise continuous control $u\colon \R^+\to\R^d$ with $\|u\|_\infty\leq \varepsilon$ such that
there is a trajectory of the system of ODEs~\eqref{eq_odectrl2}
satisfying $x(0)=x_0$ and  $x(T)=y_0$ for some $T>0$.
%(such a trajectory is clearly unique, if $V$ provides unique solvability of~\eqref{eq_odeaut1}).
\end{corollary}

\begin{proof}
Fixed an $\varepsilon>0$,
by means of a convolution with an appropriate smooth approximate identity with compact support we may find a smooth $V_\varepsilon\colon \R^d\to \R^d$
%we define $V_\varepsilon:= V * \varphi_\delta$,
%where $\varphi_\delta$ is some smooth approximate identity, choosing
%a $\delta>0$
satisfying
%so that
$\|V-V_\varepsilon\|_\infty\leq \varepsilon/2$
(this is possible because $V$ is assumed to be uniformly continuous).
Clearly, $V_\varepsilon$ is still incompressible and has vanishing mean drift.
Since now $V_\varepsilon$ is smooth, we may
apply Theorem~\ref{th_control1} to find a piecewise continuous control
$\tilde u\colon \R^+\to \R^d$ with
$\|\tilde u\|_\infty\leq \varepsilon/2$ such that the trajectory $x(\cdot)$ of the system
\[
\dot{x}= V_\varepsilon (x)+\tilde u(t)
\]
starting at $x_0\in \R^d$ arrives  at $x_1\in \R^d$ in finite time.
It suffices to take then $u(t):= V_\varepsilon(x(t))-V(x(t))+\tilde u(t)$.
\end{proof}

At last, we are able to prove the following version of the Pugh closing lemma which is the second principal result of the paper.

\begin{theorem}\label{th_Pugh1}
Let $V\in C^1(\R^d;\R^d)\cap \Lip(\R^d;\R^d)$ be a bounded
incompressible vector field with uniformly continuous first derivatives and
satisfying vanishing mean drift
condition.
Then for every  $x_0\in \R^d$ and every $\varepsilon>0$ there is a $C^1$ vector field
$Y_\varepsilon\colon\R^d \to \R^d$
 with $\|Y_\varepsilon\|_{\Lip}\leq \varepsilon$ such that $x_0$ is a periodic (or occasionally even stationary)
 point of the
system of ODEs
\begin{equation}\label{eq_odePugh2}
\dot x = V(x)+Y_\varepsilon(x).
\end{equation}
\end{theorem}

\begin{proof}
Fixed an $\varepsilon>0$, by
Theorem~\ref{th_Pugh0} we find a smooth vector field
$W\colon\R^d \to \R^d$
 with $\|W\|_\infty \leq \varepsilon/2$ such that every $x\in \R^d$ is nonwandering
 with respect to the flow defined by $V+W$.
%We choose a $p\in ((d-1)/2,d/2)$
%and
%an $\alpha>0$ so as to have $\|W\|_{C^1}\leq \varepsilon/2$, where
%$W$ is defined by~\eqref{eq_defW3}: this is possible in view of
% Lemma~\ref{lm_WCsmall} and Lemma~\ref{lm_WC1small} (when applying the latter lemma recall that the first derivatives of $V$ also are incompressible and have vanishing mean drift by Lemma~\ref{lm_md_der1}).
%again, as in the proof of Theorem~\ref{th_control1}
%% Since we have chosen $p> (d-1)/2$, then
%% \[
%% \mu(B_R(0)) =\int_{B_R(0)} \psi(x)\, dx = d\omega_d\int_0^R \frac{r^{d-1}\,dr}{(r^2+\alpha^2)^p}=o(R),
%% \]
%we have $\mu(B_R(0)) =o(R)$, and thus
%by Corollary~\ref{co_Poincare3} aplied to $V+W$ (in place of $V$)
%all the points of $\R^d=\supp\mu$ are nonwandering for the flow generated by the
%vector field $V+W$.
Observe that perturbation $W$ of the vector field $V$ is global and
does not depend on the point $x_0$ which we want to make periodic.
It suffices now to use Pugh's closing Lemma~\cite{Pugh67} to construct a local small perturbation
$\tilde W$ of $V+W$ with $\|\tilde W\|_{\Lip}\le \varepsilon/2$, so that
$x_0$ becomes periodic with respect to the flow of $V+Y$, where
$Y:=W+\tilde W$, and observe that $\|Y\|_{\Lip}\leq \varepsilon$ as claimed.
\end{proof}

\appendix

\section{Vanishing mean drift condition}\label{sec_vmd0}

We collect in this section some auxiliary results essentially depending on the vanishing mean drift condition
of a vector field involved. Here the vector field will be always assumed
at least locally Lipschitz, so that its divergence is understood in the classical pointwise (a.e.) sense.
In principle this can be further somewhat weakened (e.g. by substituting incompressibility condition in the sense of vanishing pointwise divergence by its suitable weaker analogue), though we do not pursue this direction here.

\subsection{Properties of vector fields with vanishing mean drift}

We start with the following statement.

\begin{lemma}\label{lm_md_flux1}
Suppose $V\in \Lip_{loc}(\R^d;\R^d)$ is a bounded %locally Lipschitz 
vector field, and for each $\varepsilon>0$ there is an $\ell_0>0$ such that
for every $(d-1)$-dimensional box
$Q$ of sidelength $\ell\geq \ell_0$ one has that the mean flux
\begin{equation}\label{eq_smallflux1}
\frac{1}{\ell^{d-1}}\left|\int_Q V(x)\cdot n \, d\HH^{d-1}(x)\right|
\leq \varepsilon,
\end{equation}
where $n$ stands for a normal vector to the hyperplane containing $Q$
(we will say that $V$ has vanishing mean flux).
Then $V$ has vanishing mean drift.

Vice versa, if a bounded $V\in \Lip_{loc}(\R^d;\R^d)$ has vanishing mean drift and is incompressible, then it satisfies the above property (i.e.\ has vanishing mean flux).
In particular, for incompressible bounded locally Lipschitz vector fields, having vanishing mean drift is equivalent to having vanishing mean flux.
\end{lemma}

\begin{remark} It is worth observing that incompressibility
condition is essential for a vector field $V$ having vanishing mean drift to have vanishing mean flux.
In fact, the vector field $V(x,y):=(f(x),0)$ in $\R^2$, where $f$
is a smooth function with compact support in $\R$, clearly has vanishing mean drift, but not vanishing mean flux as can be seen by computing the
flux of $V$ through one-dimensional segments $ \{(x,0)\}\times [-\ell/2, \ell/2]$ as
$\ell \to \infty$; in fact, $V$ is in general not incompressible.
\end{remark}

\begin{proof}
For a fixed $\varepsilon>0$ let
$\ell_0$ be such that the vector field
(in fact, even not necessarily incompressible)
$V$ satisfies~\eqref{eq_smallflux1} for every cube $Q$ of sidelength $\ell>\ell_0$.
Then for every $i\in \{1,\ldots, d\}$ and every $\ell>\ell_0$
one has
\begin{align*}
\left|\int_{x+ [-\ell/2,\ell/2]^d} V_i(y)\, dy\right| & =
\left|\int_{-\ell/2}^{\ell/2}\, dt
\int_{(x+ [-\ell/2,\ell/2]^d)\cap\{x_i=t\}} V_i(s)\,
d\HH^{d-1}(s)
\right|\\
&\leq \int_{-\ell/2}^{\ell/2}\, dt \left|
\int_{(x+ [-\ell/2,\ell/2]^d)\cap\{x_i=t\}} V_i(s)\,
d\HH^{d-1}(s)
\right| \\
& \leq \ell \int_{-\ell/2}^{\ell/2}\varepsilon \ell^{d-1}\, dt =\varepsilon \ell^d,
\end{align*}
i.e.\ $V$ has vanishing mean drift as claimed.
The reverse statement for incompressible $V$ is lemma~3.1 from~\cite{BurNovIv16-fish}.
\end{proof}

It seems quite intuitive that the $(d-1)$-dimensional boxes in the vanishing mean flux condition of Lemma~\ref{lm_md_flux1} may be replaced by more general
increasing sequences of sets.
We give here only two particular examples to be used later.

\begin{example}\label{ex_md_flux2}
If $V\in \Lip_{loc}(\R^d;\R^d)$ is a bounded %locally Lipschitz 
vector field
with vanishing mean flux, then
for each $\varepsilon>0$ there is an $R_0=R_0(\varepsilon)>0$ such that
for every $(d-1)$-dimensional ball (i.e.\ a ball in a $(d-1)$-dimensional affine hyperplane)
$B_R$ of radius $R\geq R_0$ one has that the mean flux
\begin{equation}\label{eq_smallflux2}
\frac{1}{\HH^{d-1}(B_R)}\left|\int_{B_R} V(x)\cdot n \, d\HH^{d-1}(x)\right|
\leq \varepsilon,
\end{equation}
where $n$ stands for a normal vector to the hyperplane containing $B_R$.
In fact, given an $\varepsilon>0$, we may cover a part of the unit $(d-1)$-dimensional ball
$B$ with $N=N(\varepsilon)$ disjoint $(d-1)$-dimensional open
boxes $Q_i\subset B$, $i=1,\ldots, N$, so that
\[
\HH^{d-1}(B\setminus\cup_{i=1}^N Q_i)\leq  \frac{\varepsilon}{2\|V\|_\infty}\HH^{d-1}(B).
\]
Letting now $x_0$ stand for the center of $B_R$, we may take an $R_0>0$ (of course, depending on $\varepsilon$)
such that
\[
\left|\int_{x_0+RQ_i} V(x)\cdot n \, d\HH^{d-1}(x)\right|
\leq \frac{\varepsilon}{2} \HH^{d-1}(RQ_i) = \frac{\varepsilon}{2} R^{d-1} \HH^{d-1}(Q_i)
\]
for all $i=1,\ldots, N$ and all $R> R_0$
in view of the vanishing mean drift condition.
We get therefore
\begin{align*}
\left|\int_{B_R} V(x)\cdot n \, d\HH^{d-1}(x)\right| & \leq
 \left|\int_{\cup_{i=1}^N (x_0+R Q_i)} V(x)\cdot n \, d\HH^{d-1}(x)\right| \\
 & \qquad + \|V\|_\infty
\HH^{d-1}(RB\setminus \cup_{i=1}^N R Q_i)\\
&\leq \frac{\varepsilon}{2} R^{d-1} \HH^{d-1}(\cup_{i=1}^N Q_i) + \frac{\varepsilon}{2} R^{d-1} \HH^{d-1}(B\setminus\cup_{i=1}^N Q_i)\\
&= \varepsilon R^{d-1} \HH^{d-1}(B)
\end{align*}
for $R>R_0$ proving the claim.
\end{example}

\begin{example}\label{ex_md_flux3}
For an $x\in \partial B_1(0)$ and $r\leq 2$ denote $D_r(x):= \partial B_1(0)\cap
B_r(x)$ (i.e.\ a ball in the natural inner metric of $\partial B_1(0)$).
If $V\in \Lip_{loc}(\R^d;\R^d)$ is a bounded incompressible
%locally Lipschitz 
vector field
with vanishing mean drift (hence with vanishing mean flux by Lemma~\ref{lm_md_flux1}), then
for each $\varepsilon>0$ there is an $R_0>0$ such that
for every
$R\geq R_0$ and every $x\in \partial B_1(0)$ one has that the mean flux
\begin{equation}\label{eq_smallflux3}
\frac{1}{\HH^{d-1}(R D_r(x))}\left|\int_{RD_r(x)} V(y)\cdot n(y) \, d\HH^{d-1}(y)\right|
\leq \varepsilon,
\end{equation}
where $n(y)$ stands for the external unit normal to $\partial B_1(0)$ at $y$.
In fact, consider the spherical cap $\Omega$ cut from the unit ball $B_1(0)$
by the hyperplane $\pi$ containing the set $\partial B_1(0)\cap
\partial B_r(x)$ (i.e.\ the relative boundary of $D_r(x)$ in  $\partial B_1(0)$), so that its boundary is the union of
$D_r(x)$ with the
$(d-1)$-dimensional closed ball $C_r:=\bar B_1(0)\cap \pi$. Letting $n$ stand for the external normal
to the boundary of $\Omega$, we get, given an $\varepsilon>0$, the estimate
\begin{align*}
\left|\int_{RD_r(x)} V(y)\cdot n(y) \, d\HH^{d-1}(y)\right|
%&= \left|-\int_{R C_r} V(y)\cdot n(y) \, d\HH^{d-1}(y)
%+\int_{R \partial \Omega} V(y)\cdot n(y) \, d\HH^{d-1}(y)\right|\\
&= \left|-\int_{R C_r} V(y)\cdot n(y) \, d\HH^{d-1}(y)
+\int_{R \Omega} \mbox{div}\, V(y)\, dy\right|\\
&= \left|\int_{R C_r} V(y)\cdot n(y) \, d\HH^{d-1}(y)\right| \\
& \leq \varepsilon
\HH^{d-1}(R C_r)
=
\varepsilon
\HH^{d-1}(C_r)  R^{d-1}\\
&\leq
 \varepsilon
\HH^{d-1}(D_r(x))  R^{d-1}
\end{align*}
for all $R>R_0$, where $R_0$ is chosen (depending on $\varepsilon$)
so that the first inequality in the above chain be satisfied (which is possible by Example~\ref{ex_md_flux2}),
and therefore
\[
\frac{1}{\HH^{d-1}(R D_r(x))}\left|\int_{RD_r(x)} V(y)\cdot n(y) \, d\HH^{d-1}(y)\right| \leq \varepsilon
\]
as claimed.
\end{example}

\begin{lemma}\label{lm_md_der1}
Suppose $V\in C^1(\R^d;\R^d)$ is a vector field
with vanishing mean drift having uniformly continuous first derivatives.
Then these derivatives also have vanishing mean drift.
\end{lemma}

\begin{proof}
For an arbitrary $i \in \{1,\ldots, d\}$ and $j \in \{1,\ldots, d\}$ we have that
\begin{equation}\label{eq_estVder1}
\begin{aligned}
\left|\frac{V_i(x + t e_j)-V_i(t)}{t} -V_{i, x_j}(x)\right|  =
\left| V_{i, x_j} (x+\theta e_j)  -V_{i, x_j}(x)\right|
\end{aligned}
\end{equation}
for some $\theta\in [0,t]$ (depending possibly on $x\in \R^d$), and hence given an $\varepsilon>0$ we may choose
a $t$ such that the right hand side of~\eqref{eq_estVder1}
does not exceed $\varepsilon/2$.
Since $V^j_t:=(V_i(x + t e_j)-V_i(t))/t$ clearly has  vanishing mean drift, then
there is an $\ell_0>0$ (depending on $t$ which is fixed) such that
for every $d$-dimensional box $Q\subset \R^d$ of side length $\ell>\ell_0$
one has
\[
\left|\int_Q V_t^j(x)\, dx\right|\leq \varepsilon \ell^d/2,
\]
and therefore
one has
\begin{align*}
\left|\int_Q V_{i,x_j}(x)\, dx\right| & \leq \left|\int_Q V_t^j(x)\, dx\right| +
 \int_Q \left| V_t^j(x)-V_{i, x_j}(x)\right|\, dx
%\\
%&
\leq  \varepsilon \ell^d/2+\varepsilon \ell^d = \varepsilon \ell^d,
\end{align*}
proving the claim.
\end{proof}

% \begin{lemma}\label{lm_meandrift_der1}%%%%% FORSE FALSA???
% If the vector field $V\in \Lip_{loc}(\R^d;\R^d)$ has vanishing mean flux, then
% $V_{x_j}$ for every $j\in \{1,\ldots, n\}$ has vanishing mean drift. In particular, if $V$ is incompressible and has vanishing mean drift, then so has $V_{x_j}$.
% \end{lemma}

% \begin{proof}
% Denoting for brevity $Q:=x+ [-\ell/2,\ell/2]^d$, we  calculate
% \begin{align*}
% \left|\int_{Q} V_{x_j}(y)\, dy\right| & =
% \left|-\int_{Q\cap\{x_i=-\ell/2\}} V(s)\,
% d\HH^{d-1}(s) +\int_{Q\cap\{x_i=\ell/2\}} V(s)\,
% d\HH^{d-1}(s)
% \right|\\
% &\leq \left|
% -\int_{Q\cap\{x_i=-\ell/2\}} V(s)\,
% d\HH^{d-1}(s)\right| +\left|\int_{Q\cap\{x_i=\ell/2\}} V(s)\,
% d\HH^{d-1}(s)
% \right| \\
% & \leq 2\varepsilon \ell^{d-1},
% \end{align*}
% whenever $L>L_0$, where $L_0=L_0(\varepsilon)$ is such
% that~\eqref{eq_smallflux1} is satisfied (for the vector field $V$).
% Thus
% \[
% \frac{1}{\ell^d}\left|\int_{x+ [-\ell/2,\ell/2]^d} V_{x_j}(y)\, dy\right| \leq 2\varepsilon/L \leq \varepsilon,
% \]
% if $L>L_0\vee 2$, proving the claim.
% \end{proof}

\begin{proposition}\label{prop_md_int0}
Let $F\subset L^1(\partial B_1(0);\HH^{d-1})$
be a compact family of functions and $V\in \Lip_{loc}(\R^d;\R^d)$ be a bounded incompressible
%locally Lipschitz 
vector field with vanishing
mean drift.
Then
\[
\int_{\partial B_1(0)} f(x) x\cdot V(\alpha x) \, d\HH^{d-1}(x)\to 0
\]
as $\alpha\to +\infty$, uniformly over $f\in F$, i.e.\
for every $\varepsilon>0$ there exists an $\alpha_0>0$ such that
\[ \left|
\int_{\partial B_1(0)} f(x) x\cdot V(\alpha x) \, d\HH^{d-1}(x)
\right|<\varepsilon
\]
for all $f\in F$, $\alpha\ge \alpha_0$.
\end{proposition}

\begin{proof}
If not, there is a sequence $\{\alpha_n\}\subset \R$,
$\lim_n \alpha_n =+\infty$, and $f_n\subset F$, such that
\[
\left|
\int_{\partial B_1(0)} f_n(x) x\cdot V(\alpha_n x) \, d\HH^{d-1}(x)
\right|\geq \varepsilon_0
\]
for some $\varepsilon_0>0$ and all $n\in \N$.
By eventually passing to a subsequence of $n$ (not relabeled) we may assume
that $f_n\to f$ in $L^1(\partial B_1(0))$ as $n\to \infty$.
But then
\begin{align*}
\left| \int_{\partial B_1(0)} f_n(x) x\cdot V(\alpha_n x) \, d\HH^{d-1}(x)
\right| & \leq \int_{\partial B_1(0)} |f_n(x)-f(x)| x\cdot V(\alpha_n x) \, d\HH^{d-1}(x) \\
&\quad + \left| \int_{\partial B_1(0)} f(x) x\cdot V(\alpha_n x) \, d\HH^{d-1}(x)
\right|\\
& \leq  \|V\|_\infty \|f_n-f\|_1 \\
&\quad +\left| \int_{\partial B_1(0)} f(x) x\cdot V(\alpha_n x) \, d\HH^{d-1}(x)
\right| \to 0
\end{align*}
since the latter integral is vanishing by Lemma~\ref{lm_md_int0}, a contradiction.
\end{proof}

The following lemma has been used in the above proof.

\begin{lemma}\label{lm_md_int0}
Let $f\in L^1(\partial B_1(0);\HH^{d-1})$ and $V\in \Lip_{loc}(\R^d;\R^d)$ be a bounded incompressible
%locally Lipschitz 
vector field with vanishing
mean drift.
Then
\[
\int_{\partial B_1(0)} f(x) x\cdot V(\alpha x) \, d\HH^{d-1}(x)\to 0
\]
as $\alpha\to +\infty$.
\end{lemma}

\begin{proof}
It is enough to prove the statement for $f$ from a family of functions having dense linear span in
$L^1(\partial B_1(0);\HH^{d-1})$, in particular, for $f$
% a function of the
% form
% $f=\sum_{i=1}^N a_i \mathbf{1}_{D_{r_i}(x_i)}$
% for some $N\in \N$, $a_i\in \R$, $r_i\in (0,2]$, $x_i\in \partial B_1(0)$,
% the sets $D_{r_i}(x_i)\subset \partial B_r(0)$ being defined in Example~\ref{ex_md_flux3}. Therefore it is enough to show it for $f$
just a characteristic function of the form
$f=\mathbf{1}_{D_r(x)}$ for some $r\in (0,2]$, $x\in \partial B_1(0)$, where
$D_r(x)$ is defined in Example~\ref{ex_md_flux3}. The claim for this case follows then from the change of variables
\[
\int_{D_{r}(x)} y\cdot V(\alpha y) \, d\HH^{d-1}(y)=\frac{1}{\alpha^{d-1}}\int_{\alpha D_{r}(x)}
V(y)\cdot n(y)\, d\HH^{d-1}(y),
\]
where $n$ stands for the external normal to $\partial B_1(0)$, and from Example~\ref{ex_md_flux3}.
\end{proof}

\subsection{Auxiliary estimates of integral operators}

We will also need the following technical assertions on estimates of integral operators involving vector fields with vanishing mean drift.

\begin{lemma}\label{lm_NPsmall2}
Let $p \in ((d-1)/2, d/2)$, and
$K\colon \R^d\times\R^d\to \R$ be a function locally Lipschitz outside
the diagonal $\{(x,y)\in \R^d\times\R^d\colon x=y\}$
and uniformly continuous outside every
$R$-neighborhood of the diagonal (for every $R>0$),
with
\[
K(x,y)=\frac{A(x,y)}{|x-y|^k}, \quad k< d,
\]
for all $(x,y)\in \R^d\times\R^d$,
for some $A\in L^\infty(\R^d\times\R^d)$ continuous outside the diagonal.
% and $0<k<d+1$.
 Suppose that $V\in \Lip_{loc}(\R^d;\R^d)$ is a bounded incompressible
vector field with vanishing mean drift.
Then, with the notation $n(x):=x/|x|$ one has
\begin{equation}\label{eq_2WA4}
\Psi(x):=\int_{\R^d} K(n(x),y)
\frac{y\cdot V(\alpha |x| y)}{(|y|^2+|x|^{-2})^{p+1}}\, dy\to 0
\end{equation}
as $\alpha\to +\infty$ uniformly in $x\not \in
B_1(0)\subset \R^d$.
\end{lemma}

\begin{proof}
Without loss of generality (up to a small increase of $k$) we may assume
$A$ to be continuous over the whole $\R^d\times \R^d$.
Let $x\in B_1^c(0)\subset \R^d$ and choose an arbitrary $\varepsilon>0$.
% and denote for brevity
%\[
%C(p):= \sup_{y\in \R^d} \frac{|y|}{(|y|^2+1)^{p+1}}
%\]
%(note that $C(p)<+\infty$ because $p>-1/2$).

Fixed and arbitrary $\rho\in (0,1)$, the function
$F\colon B_\rho(0)\to \R^d$ defined by
\[
F(y):=K(n(x),y)-K(n(x),0)
\]
is Lipschitz (because of local Lipschitz continuity of $K$ outside the diagonal), while
clearly, $F(0)=0$, so that
one has in particular $|F(y)|\leq C |y|$ for some $C=C(\rho)>0$ (independent on $x$)
whenever $|y|\leq \rho$. Therefore,
\begin{align*}
\left| \int_{B_\rho(0)} \right. &
\left. \left(K(n(x),y)-K(n(x),0)\right)\frac{y\cdot V(\alpha |x| y)}{(|y|^2+|x|^{-2})^{p+1}}\, dy\right| \\
&     \leq %C(p)
C \|V\|_\infty
\int_{B_\rho(0)}
\frac{|y|^2}{(|y|^2+|x|^{-2})^{p+1}}\, dy\\
&\leq %C(p)
C \|V\|_\infty
\int_{B_\rho(0)}
\frac{|y|^2}{|y|^{2(p+1)}}\, dy
%\\
%&
=%C(p)
C d\omega_d \|V\|_\infty \int_0^{\rho}
\frac{\,dr}{r^{2p-d+1}},
\end{align*}
and recalling that $p< d/2$ we conclude that one can choose a $\rho>0$ so that
\[
\left| \int_{B_\rho(0)}\left(K(n(x),y)-K(n(x),0)\right)
\frac{y\cdot V(\alpha |x| y)}{(|y|^2+|x|^{-2})^{p+1}}\, dy\right|\leq \varepsilon/3.
\]
But by Lemma~\ref{lm_2WC1div0} one has
\begin{align*}
 \int_{B_\rho(0)}\left(K(n(x),y)-K(n(x),0)\right)  &
\frac{y\cdot V(\alpha |x| y)}{(|y|^2+|x|^{-2})^{p+1}}\, dy\\
&= \int_{B_\rho(0)} K(n(x),y)
\frac{y\cdot V(\alpha |x| y)}{(|y|^2+|x|^{-2})^{p+1}}\, dy,
\end{align*}
and hence
\begin{equation}\label{eq_2WA5}
 \left| \int_{B_\rho(0)} K(n(x),y)
\frac{y\cdot V(\alpha |x| y)}{(|y|^2+|x|^{-2})^{p+1}}\, dy\right|\leq \varepsilon/3.
\end{equation}

Choose now an $R\in (0,1-\rho)$ such that
\[
\int_{B_R(n(x))}\frac{\, dy}{|n(x)-y|^k}  < \delta,
\]
for some $\delta>0$ to be chosen later (it is only here that we use the
assumption on $k$).
Denoting
\[
K_R(n(x),y):= \left\{
 \begin{array}{rl}
 K(n(x),y), & y\in B_R^c(n(x)),\\
A(n(x),y)/R^k, & y\in B_R(n(x)),
 \end{array}
 \right.
\]
we get that $(x,y)\mapsto K_R(n(x),y)$ is bounded and uniformly
continuous in $B_1^c(0)\times \R^d$.
%and $\|\psi_R\|_\infty\leq 1/R^{k-1}$.
Now, recalling that $\rho< |y|<1$ when $y\in B_R(n(x))$, we get
\begin{align*}
\left| \int_{B_\rho^c(0)}\right. & \left. K(n(x),y)
\frac{y\cdot V(\alpha |x| y)}{(|y|^2+|x|^{-2})^{p+1}}\, dy\right| \\
 & \leq
 \left| \int_{B_\rho^c(0)\cap B_R^c(n(x))} K(n(x),y)
 \frac{y\cdot V(\alpha |x| y)}{(|y|^2+|x|^{-2})^{p+1}}\, dy\right|
  +\frac{\delta \|V\|_\infty\|A\|_\infty}{\rho^{2(p+1)}}\\
   & \leq
 \left| \int_{B_\rho^c(0)} K_R (n(x),y) \frac{y\cdot V(\alpha |x| y)}{(|y|^2+|x|^{-2})^{p+1}}\, dy\right|\\
 & \qquad +  \left| \int_{B_R(0)} \frac{A(n(x),y)}{R^k} \frac{y\cdot V(\alpha |x| y)}{(|y|^2+|x|^{-2})^{p+1}}\, dy\right| +\frac{\delta \|V\|_\infty\|A\|_\infty}{\rho^{2(p+1)}}\\
 &     \leq
 \left| \int_{B_\rho^c(0)} K_R (n(x),y) \frac{y\cdot V(\alpha |x| y)}{(|y|^2+|x|^{-2})^{p+1}}\, dy\right|\\
 & \qquad + \omega_d R^{d-k}\|V\|_\infty\|A\|_\infty \frac{1}{\rho^{2(p+1)}} +
 \frac{\delta \|V\|_\infty\|A\|_\infty}{\rho^{2(p+1)}},
\end{align*}
so that choosing $\delta$ (depending on $\rho$ and $\varepsilon$) and
$R$ (depending on $\delta$, $\rho$ and $\varepsilon$) sufficiently small we will have
\begin{equation}\label{eq_2WA6}
\begin{aligned}
\left| \int_{B_\rho^c(0)} K(n(x),y) \right. &\left.
\frac{y\cdot V(\alpha |x| y)}{(|y|^2+|x|^{-2})^{p+1}}\, dy\right| \\
&\leq
\left| \int_{B_\rho^c(0)} K_R (n(x),y) \frac{y\cdot V(\alpha |x| y)}{(|y|^2+|x|^{-2})^{p+1}}\, dy\right| + \varepsilon/3.
\end{aligned}
\end{equation}

Finally,
\begin{equation}\label{eq_2WA6a}
\begin{aligned}
\int_{B_\rho^c(0)} K_R & (n(x),y)
   \frac{y\cdot V(\alpha |x| y)}{(|y|^2+|x|^{-2})^{p+1}}\, dy \\
 & = \int_\rho^{+\infty} \frac{r^d\,dr}{(r^2+|x|^{-2})^{p+1}} \int_{\partial B_1(0)}
K_R (n(x),r s) s \cdot V(\alpha |x| rs)\,d\HH^{d-1}(s).
\end{aligned}
\end{equation}
Recalling that the family of functions $\{s\in \partial B_1(0)
\mapsto K_R (n(x),r s)\colon r\geq \rho, x\in B_\rho^c(0) \}$ on $\partial B_1(0)$
 is bounded and equicontinuous, we get from
Proposition~\ref{prop_md_int0} that
\begin{align*}
\int_{\partial B_1(0)} K_R(n(x),rs)  s\cdot V(\beta s)\,d\HH^{d-1}(s) \to 0
\end{align*}
uniformly in $x\in B_1^c(0)$ and $r\geq\rho$ as $\beta\to +\infty$, so that
in particular
\begin{align*}
\int_{\partial B_1(0)} K_R(n(x),rs)  s\cdot V(\alpha |x|  r s)\,d\HH^{d-1}(s) \to 0
\end{align*}
uniformly in $x\in B_1^c(0)$ and $r\geq \rho$ as $\alpha\to +\infty$.
Thus, observing that
\[
\frac{r^d}{(r^2+|x|^{-2})^{p+1}}\leq \frac{r^d}{(r^2+1)^{p+1}}
\]
for $x\in B_1^c(0)$ and $r\geq 0$, the function on the right-hand side of the above inequality being integrable over $(0,+\infty)$ (because $p>(d-1)/2$),
we get from~\eqref{eq_2WA6a}
the existence of some $\bar\alpha>0$
(independent of $x$) such that
the estimate
\begin{equation}\label{eq_2WA7}
\begin{aligned}
\left| \int_{B_\rho^c(0)} K_R (n(x),y) \right. &\left.
   \frac{y\cdot V(\alpha |x| y)}{(|y|^2+|x|^{-2})^{p+1}}\, dy \right|
&\leq
 \varepsilon/3,
\end{aligned}
\end{equation}
holds for all $x\in B_1^c(0)$
once $\alpha\geq \bar\alpha$.
%for some $\bar\alpha>0$
%independent of $x$. %is sufficiently large independent on $x$.
Combining the estimates~\eqref{eq_2WA5},~\eqref{eq_2WA6} and~\eqref{eq_2WA7}
we arrive from~\eqref{eq_2WA4} to $|\Psi(x)|\leq \varepsilon$
for %sufficiently large
$\alpha\geq \bar\alpha$, with $\bar\alpha>0$
independent on
$x\in B_1(0)^c$, as claimed.
\end{proof}

The following lemma has been used in the above proof.

\begin{lemma}\label{lm_2WC1div0}
For an incompressible vector field $V\in \Lip_{loc}(\R^d;\R^d)$ one has
\[
\int_{B_\rho(0)} \frac{y\cdot V(\alpha|x|y)}{(|y|^2+|x|^{-2})^p}=0.
\]
for every $\rho>0$ and every $\alpha>0$.
\end{lemma}

\begin{proof}
We write
\[
\int_{B_\rho(0)} \frac{y\cdot V(\alpha|x|y)}{(|y|^2+|x|^{-2})^p} =
\int_0^\rho \frac{r^k\,dr}{(r^2+|x|^{-2})^p} \int_{\partial B_1(0)}
s \cdot V(\alpha|x|rs)\, d\HH^{d-1}(s),
\]
and recalling that $V$ is divergence-free, and hence so is the vector field
$y\mapsto   V(\alpha |x| r y)$ (with $\alpha$, $r$ and $x$ fixed), one has
\[\int_{\partial B_1(0)}
s \cdot V(\alpha|x|rs)\, d\HH^{d-1}(s)=0
\]
implying the thesis.
\end{proof}

At last we need the following computation.

\begin{lemma}\label{lm_NPsmall1}
Let $p >(d-1)/2$, and
$K\colon \R^d\times\R^d\to \R$ be a function uniformly continuous outside every
$R$-neighborhood of the diagonal $\{(x,y)\in \R^d\times\R^d\colon x=y\}$ (for every $R>0$), with
\[
K(x,y)=\frac{A(x,y)}{|x-y|^k}, \quad k< d,
\]
for some $A\in L^\infty(\R^d\times\R^d)$ continuous outside the diagonal and
% and $0<k<d+1$.
for all $(x,y)\in \R^d\times\R^d$.
 Suppose that $V\in \Lip_{loc}(\R^d;\R^d)$ is a bounded incompressible
vector field with vanishing mean drift.
Then
\begin{equation}\label{eq_2WA0}
\tilde \Psi(x):=\int_{\R^d} K(x,y)
\frac{y\cdot V(\alpha y)}{(|y|^2+1)^{p+1}}\, dy\to 0
\end{equation}
as $\alpha\to +\infty$ uniformly in $x\in \bar B_1(0)\subset \R^d$.
\end{lemma}

\begin{proof}
As in the proof of Lemma~\ref{lm_NPsmall2} without loss of generality (up to a small increase of $k$) we assume
$A$ to be continuous over the whole $\R^d\times \R^d$.
Consider an arbitrary $x\in \bar B_1(0)\subset \R^d$ and an arbitrary $\varepsilon>0$.
%Let $x\in B_1^c(0)\subset \R^d$, choose an arbitrary $\varepsilon>0$ and
Denote for brevity
\[
C(p):= \sup_{y\in \R^d} \frac{|y|}{(|y|^2+1)^{p+1}}
\]
(note that $C(p)<+\infty$ because $p>-1/2$).
%We retain the notation on the constant $C(p)$
%used in the proof of Lemma~\ref{lm_NPsmall2}.
Finding a $\rho\in [0,1]$ such that
\begin{equation}\label{eq_2Wacut1}
\int_{B_\rho(x)}\frac{\, dy}{|x-y|^k}  < \delta,
\end{equation}
for some $\delta>0$ to be chosen later,
we get
\begin{equation}\label{eq_2WA2}
\begin{aligned}
|\tilde \Psi(x)| &
\leq \left( \|A\|_\infty \omega_d\rho^d \delta C(p) \|V\|_\infty + |I(\alpha,\rho)|\right),\quad
\mbox{where }\\
& I(\alpha, \rho):=\int_{B_\rho^c(x)} K(x,y)
\frac{y\cdot V(\alpha y)}{(|y|^2+1)^{p+1}}\, dy.
\end{aligned}
\end{equation}
Denoting
\begin{equation}\label{eq_defphi1}
K_\rho(x,y):= \left\{
\begin{array}{rl}
K(x,y), & y\in B_\rho^c(x),\\
A(x,y)/\rho^k, & y\in B_\rho(x),
\end{array}
\right.
\end{equation}
we get that $K_\rho$ is uniformly continuous and
bounded in $\R^d\times\R^d$.
We have now
\begin{equation}\label{eq_2WA2a1}
\begin{aligned}
|I(\alpha,\rho)| &\leq \left|\int_{\R^d} K_\rho(x,y)
\frac{y\cdot V(\alpha y)}{(|y|^2+1)^{p+1}}\, dy\right| +
 \left|\int_{B_\rho(x)} \frac{A(x,y)}{\rho^k}
\frac{y\cdot V(\alpha y)}{(|y|^2+1)^{p+1}}\, dy\right|\\
&\leq  \left|\int_{\R^d} K_\rho(x,y)
\frac{y\cdot V(\alpha y)}{(|y|^2+1)^{p+1}}\, dy\right| +\|A\|_\infty C(p)\|V\|_\infty
\omega_d\rho^{d-k}.
\end{aligned}
\end{equation}
Recall that
\begin{align*}
\int_{\R^d} K_\rho(x,y) &
\frac{y\cdot V(\alpha y)}{(|y|^2+1)^{p+1}}\, dy \\
& =
\int_0^{+\infty} \frac{r^d}{(r^2+1)^{p+1}} \, dr
\int_{\partial B_1(0)} K_\rho(x,rs)s\cdot V(\alpha r s)\,d\HH^{d-1}(s).
\end{align*}
We choose $\delta>0$ and $\rho>0$ to be sufficiently small (depending on $\varepsilon$) so as to have
\begin{align*}
\|A\|_\infty \omega_d \rho^d \delta C(p) \|V\|_\infty < \varepsilon/5,\quad
\|A\|_\infty C(p)\|V\|_\infty
\omega_d\rho^{d-k}< \varepsilon /5,
\end{align*}
so that in view of~\eqref{eq_2WA2a1} the estimate~\eqref{eq_2WA2} becomes
\begin{equation}\label{eq_2WA3}
 |\tilde \Psi(x)|\leq 2\varepsilon/5 + \left|\int_{\R^d} K_\rho(x,y)
\frac{y\cdot V(\alpha y)}{(|y|^2+1)^{p+1}}\, dy\right|.
\end{equation}

We choose now $r_0>0$ and $r_1>r_0$ (depending on $\varepsilon$ and $\rho$)
such that
\begin{align*}
\int_0^{r_0} \frac{r^d}{(r^2+1)^{p+1}} \, dr &
\int_{\partial B_1(0)} \left|K_\rho(x,rs)s\cdot V(\alpha r s)\right|\,d\HH^{d-1}(s)
\\
&\leq
\|K_\rho\|_\infty d \omega_d\|V\|_\infty \int_0^{r_0} \frac{r^d}{(r^2+1)^{p+1}} \, dr
\leq \varepsilon/5, \\
\int_{r_1}^{+\infty} \frac{r^d}{(r^2+1)^{p+1}} \, dr &
\int_{\partial B_1(0)} \left|\varphi_\rho(x,rs)s\cdot V(\alpha r s)\right|\,d\HH^{d-1}(s)
\\
&\leq \|K_\rho\|_\infty d \omega_d\|V\|_\infty \int_{r_1}^{+\infty}
 \frac{r^d}{(r^2+1)^{p+1}} \, dr \leq\varepsilon/5,
\end{align*}
(recall that we assumed $p> (d-1)/2$).
Since by Proposition~\ref{prop_md_int0} one has
\begin{align*}
\int_{\partial B_1(0)} K_\rho(x,rs)  s\cdot V(\alpha r s)\,d\HH^{d-1}(s) \to 0
\end{align*}
uniformly in $x\in B_1(0)$ as $\alpha\to +\infty$ for all $r\in [r_0, r_1]$,
we get
\[
\left|\int_{\R^d} K_\rho(x,y)
\frac{y\cdot V(\alpha y)}{(|y|^2+1)^{p+1}}\, dy\right| \leq 3\varepsilon/5,
\]
for all %sufficiently large
$\alpha\geq \bar\alpha$ (with some $\bar\alpha>0$ independent of $x\in B_1(0)$)
and hence, by~\eqref{eq_2WA3},
$|\tilde \Psi(x)| \leq\varepsilon$ for such $\alpha$ as claimed.
\end{proof}

\section{Estimates on the corrector}\label{sec_estcorr1}

\subsection{Uniform estimates}

We show that with the appropriate choice of the parameters the vector field
$W$ defined by the formula~\eqref{eq_defW3} can be made arbitrarily small
in supremum norm, namely, that the following lemma is valid.

\begin{lemma}\label{lm_WCsmall}
Let $p\in ((d-1)/2, d/2)$. Suppose that $V\in \Lip_{loc}(\R^d;\R^d)$ is a bounded incompressible
vector field with vanishing mean drift.
Then, given an $\varepsilon>0$, there is an $\bar\alpha=\bar\alpha(p,\varepsilon)$ such that
\[
\|W\|_\infty \leq \varepsilon\quad \mbox{
for every $\alpha> \bar\alpha$}.
\]
\end{lemma}

\begin{proof}
We calculate
\begin{equation}\label{eq_WA1}
\begin{aligned}
W_\alpha(x):= W(\alpha x)& =2p c_d\alpha^{2p} (|x|^2+1)^p
\int_{\R^d}\frac{\alpha x-y}{|\alpha x-y|^d}
\frac{y\cdot V(y)}{(|y|^2+\alpha^2)^{p+1}}\, dy\\
&
=2p c_d(|x|^2+1)^p \int_{\R^d}\frac{x-y}{|x-y|^d}
\frac{y\cdot V(\alpha y)}{(|y|^2+1)^{p+1}}\, dy,
\end{aligned}
\end{equation}
and note that $\|W\|_\infty=\|W_\alpha\|_\infty$, so that we may
estimate the latter.

{\sc Case 1}: $|x|\leq 1$. The desired estimate
follows immediately from Lemma~\ref{lm_NPsmall1}
applied with
\[
K(x,y):=\frac{x_i-y_i}{|x-y|^d}
\]
(i.e.\ $A(x,y):=(x_i-y_i)/|x-y)|$, $k=d-1$) for every $i=1,\ldots, d$.

{\sc Case 2}: $|x|> 1$.
In this case, denoting $n(x):= x/|x|$, we further calculate
\begin{equation}\label{eq_WA4}
\begin{aligned}
W_\alpha(x)& = 2p c_d(1+ |x|^{-2})^p |x|^{2p}\int_{\R^d}\frac{n(x)|x|-y}{|n(x)|x|-y|^d}
\frac{y\cdot V(\alpha y)}{(|y|^2+1)^{p+1}}\, dy\\
&= 2p c_d(1+ |x|^{-2})^p \int_{\R^d}\frac{n(x)-y}{|n(x)-y|^d}
\frac{y\cdot V(\alpha |x| y)}{(|y|^2+|x|^{-2})^{p+1}}\, dy,
%\\
%&=  2p c_d(1+ |x|^{-2})^p \int_{\R^d}\left(\frac{n(x)-y}{|n(x)-y|^d} - n(x)\right)
%\frac{y\cdot V(\alpha |x| y)}{(|y|^2+|x|^{-2})^{p+1}}\, dy,
\end{aligned}
\end{equation}
%the latter equality being due to Lemma~\ref{lm_WC1div0}.
so that the necessary estimate follows immediately from Lemma~\ref{lm_NPsmall2}
applied with the same data as in Case~1.
\end{proof}

\begin{corollary}\label{co_Wdivsmall}
Under conditions of Lemma~\ref{lm_WCsmall} one has that
for every $\varepsilon>0$, there is an $\bar\alpha=\bar\alpha(p,\varepsilon)$ such that
\[
\|\mathrm{div}\, W \|_\infty \leq \varepsilon\quad \mbox{
for every $\alpha> \bar\alpha$}.
\]
\end{corollary}

\begin{proof}
From~\eqref{eq_defW1} one has
\[
\psi \mbox{div}\, W = \mbox{div}\, (\psi W)- \nabla \psi\cdot W= -\nabla \psi\cdot V- \nabla \psi\cdot W,
\]
so that
\begin{align*}
\mbox{div}\, W = -\nabla \log \psi\cdot (V+ W) = 2p \frac{|x|}{|x|^2+\alpha^2}\frac{x}{|x|}\cdot (V+ W).
\end{align*}
This gives
\begin{align*}
|\mbox{div}\, W| \leq  \frac{2p}{\alpha} \frac{|x/\alpha|}{|x/\alpha|^2+1}(\|V\|_\infty+ \|W\|_\infty)\leq
\frac{p}{\alpha} (\|V\|_\infty+ \|W\|_\infty),
\end{align*}
since $|t|/(t^2+1)\leq 1/2$, concluding the proof.
\end{proof}

\subsection{Estimates on the derivatives}

We show now that if $V$ is sufficiently smooth, then the derivatives
of $W$ can also be made as small as desired with the appropriate choice of parameters.

\begin{lemma}\label{lm_WC1small}
Let $p\in ((d-1)/2,d/2)$.
%$p> (d-1)/2$.
Assume that $V\in
\Lip(\R^d;\R^d)$ is incompressible vector field with vanishing mean drift, and
all $V_{x_j}$, $j=1, \ldots, d$, are locally Lipschitz still having vanishing mean drift.
Given an $\varepsilon>0$, there is an $\bar\alpha=\bar\alpha(p,\varepsilon)$ such that
\[
\|W_{x_j}\|_\infty \leq \varepsilon\quad \mbox{for every $\alpha> \bar\alpha$}
\]
for every $j=1,\ldots, d$.
\end{lemma}

\begin{proof}
Fix an arbitrary $j\in \{1,\ldots, d\}$. Rewriting~\eqref{eq_defW3} in the form
\begin{align*}
W(x) & =
 2p(-1)^d c_d (|x|^2+\alpha^2)^p\int_{\R^d} \frac{z}{|z|^d} \frac{ (x-z)\cdot V(x-z)}{(|x-z|^2+\alpha^2)^{p+1}}
 \, dz,
\end{align*}
%and denoting for brevity $C_{p,d}:= 2p c_d$,
we get the following relationships
\begin{eqnarray}
\label{eq_Wder1}
W_{x_j}(x) = Z_1(x) +Z_2(x)+ Z_3(x), \quad\mbox{where}\\
\label{eq_Z1}
\begin{array}{rl}
Z_1(x) & = 4p^2 (-1)^dc_d x_j (|x|^2+\alpha^2)^{p-1}
\displaystyle\int_{\R^d}
\frac{z}{|z|^d} \frac{(x-z)\cdot V(x-z)}{(|x-z|^2+\alpha^2)^p}\, dz \\
&
= 4p^2 c_d x_j (|x|^2+\alpha^2)^{p-1}
\displaystyle\int_{\R^d} \frac{x-y}{|x-y|^d} \frac{y\cdot V(y)}{(|y|^2+\alpha^2)^p}\, dy \\
&=
\displaystyle\frac{2p x_j}{|x|^2+\alpha^2} W(x),
\end{array}\\
\label{eq_Z2}
\begin{array}{rl}
Z_2(x) & = 2p(-1)^d c_d
 (|x|^2+\alpha^2)^{p}
\displaystyle\int_{\R^d} \dfrac{z}{|z|^d} \frac{(x-z)\cdot V_{x_j}(x-z)}{(|x-z|^2+\alpha^2)^{p+1}}  \, dz\\
&=2p c_d(|x|^2+\alpha^2)^{p}
\displaystyle\int\limits_{\R^d} \frac{x-y}{|x-y|^d} \frac{ y\cdot V_{y_j}(y)}{(|y|^2+\alpha^2)^{p+1}}\, dy,
\end{array}\\
\label{eq_Z3}
\begin{array}{rl}
Z_3(x) & = 2p(-1)^dc_d(|x|^2+\alpha^2)^{p}
\displaystyle\int\limits_{\R^d} \dfrac{z}{|z|^d} V(x-z)\cdot\frac{\partial}{\partial x_i}
\dfrac{x-z}{(|x-z|^2+\alpha^2)^{p+1}}\, dz\\
&
= 2p c_d(|x|^2+\alpha^2)^{p}
\displaystyle\int\limits_{\R^d} \dfrac{x-y}{|x-y|^d} V(y) \cdot \dfrac{\partial}{\partial y_i}
\dfrac{y}{(|y|^2+\alpha^2)^{p+1}}\, dy.
\end{array}
\end{eqnarray}
We now estimate separately the three terms $Z_1$, $Z_2$ and $Z_3$.

{\sc Estimate of $Z_1$}. Since
the function $x\mapsto x_j/(|x|^2+\alpha^2)$ is uniformly bounded over $\R^d$,
then from~\eqref{eq_Z1}
by Lemma~\ref{lm_WCsmall} we get that $\|Z_1\|_\infty\leq \varepsilon/3$ once
$\alpha$ is sufficiently large (depending on $p$ and $\varepsilon$).

{\sc Estimate of $Z_2$}. From~\eqref{eq_Z2} we see that the expression for
$Z_2$ %, modulo a different sign,
is exactly the same as the definition~\eqref{eq_defW3} of $W$, but with
$V_{x_j}$ instead of $V$. Since $V_{x_j}$ is still bounded, incompressible
and has vanishing mean drift, %(the latter by Lemma~\ref{lm_meandrift_der1}),
then again
by Lemma~\ref{lm_WCsmall} we get that $\|Z_2\|_\infty\leq \varepsilon/3$ once
$\alpha$ is sufficiently large (depending on $p$ and $\varepsilon$).

{\sc Estimate of $Z_3$}.
We let
\begin{equation}\label{eq_barZ3}
Z_{3,\alpha}(x):= Z_3(\alpha x)=\frac{2pc_d}{\alpha} (|x|^2+1)^p \int_{\R^d}
\frac{x-y}{|x-y|^d} V(\alpha y) \cdot \frac{\partial}{\partial y_j}
\frac{y}{(|y|^2+1)^{p+1}}\, dy,
\end{equation}
and note that $\|Z_3\|_\infty=\|Z_{3,\alpha}\|_\infty$, so that we may estimate the latter. As in the proof of Lemma~\ref{lm_WCsmall} we separate this estimate in two
cases.

\emph{Case $|x|\leq 1$}. In this case we repeat line by line the %Case~1 of the
proof of Lemma~\ref{lm_NPsmall1}
with
\[
K(x,y):=\frac{x_i-y_i}{|x-y|^d}
\]
(i.e.\ $A(x,y):=(x_i-y_i)/|x-y)|$, $k=d-1$) for every $i=1,\ldots, d$, but with the vector field
$\frac{\partial}{\partial y_j}
\frac{y}{(|y|^2+1)^{p+1}}$ in place of just $y/(|y|^2+1)^{p+1}$, observing that
this vector field is still uniformly bounded by a constant
depending only on $p$ that we will call $C(p)$.
Then choosing an arbitrary $\delta>0$ and a  $\rho>0$ so as to
satisfy~\eqref{eq_2Wacut1},
instead of~\eqref{eq_2WA2} and~\eqref{eq_2WA2a1} we will get, with
the notation of $\varphi_\rho$ defined in~\eqref{eq_defphi1}, the estimate
\begin{equation}\label{eq_Z3a1}
\begin{aligned}
|Z_{3,\alpha}(x)| &
\leq \frac{2pc_d 2^p}{\alpha} \delta C(p) \|V\|_\infty +  \frac{2pc_d 2^p}{\alpha}C(p)\|V\|_\infty
\omega_d\rho\\
& \qquad+  \frac{2pc_d 2^p}{\alpha}
\left|\int_{\R^d} \varphi_\rho(x,y)
V(\alpha y) \cdot\frac{\partial}{\partial y_j}
\frac{y}{(|y|^2+1)^{p+1}}\, dy\right|.
\end{aligned}
\end{equation}
Writing
\begin{equation}\label{eq_Z3a2}
\begin{aligned}\int_{\R^d} \varphi_\rho &(x,y)
V(\alpha y) \cdot\frac{\partial}{\partial y_j}
\frac{y}{(|y|^2+1)^{p+1}}\, dy\\
& =
\int_0^{+\infty} \frac{r^{d-1}}{(r^2+1)^{p+1}} \, dr
\int_{\partial B_1(0)} \varphi_\rho(x,rs) V_j (\alpha rs)
\,d\HH^{d-1}(s)\\
&\quad - 2(p+1) \int_0^{+\infty} \frac{r^{d+1}}{(r^2+1)^{p+2}} \, dr
\int_{\partial B_1(0)} \varphi_\rho(x,rs) s_j
s\cdot V(\alpha r s)
\,d\HH^{d-1}(s),
\end{aligned}
\end{equation}
and estimating
\begin{align*}
\left|\int_0^{+\infty} \frac{r^{d-1}}{(r^2+1)^{p+1}} \, dr \right. &\left.
\int_{\partial B_1(0)} \varphi_\rho(x,rs) V_j (\alpha rs)
\,d\HH^{d-1}(s)\right| \\
&  \leq d\omega_d \frac{1}{\rho^{d-1}} \|V\|_\infty
 \int_0^{+\infty} \frac{r^{d-1}}{(r^2+1)^{p+1}} \, dr,\\
\left|\int_0^{+\infty} \frac{r^{d+1}}{(r^2+1)^{p+2}} \, dr\right. &\left.
 \int_{\partial B_1(0)} \varphi_\rho(x,rs) s_j
 s\cdot V(\alpha r s)
 \,d\HH^{d-1}(s)\right| \\
& \leq d\omega_d \frac{1}{\rho^{d-1}} \|V\|_\infty \int_0^{+\infty} \frac{r^{d+1}}{(r^2+1)^{p+2}} \, dr,
\end{align*}
the integrals on the right-hand sides of the above inequalities
being convergent (because $p> d/2-1$), we get that the integral on the left-hand side of~\eqref{eq_Z3a2} is estimated by a constant depending only on $d$, $p$, $\|V\|_\infty$ and $\rho$. Recalling that $\alpha$ is in the denominator
in the right-hand side of~\eqref{eq_Z3a1}, we get from~\eqref{eq_Z3a1} together
with~\eqref{eq_Z3a2} that $|Z_{3,\alpha}(x)|\leq \varepsilon/3$ for $\alpha$ sufficiently large (depending only on $p$ and $\varepsilon$).

\emph{Case $|x|> 1$}. We rewrite
 \begin{equation}\label{eq_barZ3a3}
Z_{3,\alpha}(x)=\frac{2pc_d}{\alpha|x|} (1+ |x|^{-2})^p \int_{\R^d}
\frac{n(x)-y}{|n(x)-y|^d} V(\alpha |x| y) \cdot \frac{\partial}{\partial y_j}
\frac{y}{(|y|^2+|x|^{-2})^{p+1}}\, dy,
\end{equation}
where $n(x):= x/|x|$.
Let now
\[
\begin{aligned}
\Lambda_1 (x)&:=
\frac{2pc_d}{\alpha|x|} (1+ |x|^{-2})^p \\
& \qquad\int_{\R^d}
\left( \frac{n(x)-y}{|n(x)-y|^d} - \frac{n(x)-y}{|n(x)-y|^{d-1}}\right)
 V(\alpha |x| y) \cdot \frac{\partial}{\partial y_j}
\frac{y}{(|y|^2+|x|^{-2})^{p+1}}\, dy\\
\Lambda_2 (x)&:=
\frac{2pc_d}{\alpha|x|} (1+ |x|^{-2})^p \int_{\R^d}
\frac{n(x)-y}{|n(x)-y|^{d-1}}
 V(\alpha |x| y) \cdot \frac{\partial}{\partial y_j}
\frac{y}{(|y|^2+|x|^{-2})^{p+1}}\, dy,
\end{aligned}
\]
so that $Z_{3,\alpha}(x)= \Lambda_1(x)+\Lambda_2(x)$.
We estimate the terms $\Lambda_1$ and $\Lambda_2$ separately.

To estimate $\Lambda_1$, consider the function $f\colon \partial B_1(0)\times \R^d\to \R^d$ defined by
\begin{align*}
f(x,y)&:=\frac{x-y}{|x-y|^d} - \frac{x-y}{|x-y|^{d-1}}.
\end{align*}
Since
$\left|f_y(x,y)\right|\leq C$
for some $C=C(\rho)>0$ (independent on $x$ and $y$) whenever $|y|\leq \rho<1$ and $|x|=1$, then
for the function
$F\colon B_1(0)\to \R^d$ defined as $F(y):= f(n(x),y)$
one clearly has $F(0)=0$ and
$|F(y)|\leq C |y|$
whenever $|y|\leq \rho$, $|x|>1$ for some $\rho\in (0,1)$.
Therefore,
one gets
\begin{equation}\label{eq_estL1a}
\begin{aligned}
|\Lambda_1(x)| &\leq \frac{2pc_d}{\alpha} 2^p C \|V\|_\infty
 \int_{B_\rho(0)} |y|\cdot \left| \frac{\partial}{\partial y_j}
\frac{y}{(|y|^2+|x|^{-2})^{p+1}}\right|\, dy
%\\
%& \quad
+ \frac{2pc_d}{\alpha}2^p \left| J_{\alpha}\right|,\\
&\leq \frac{2pc_d}{\alpha} 2^pC \|V\|_\infty C(p)\omega_d\rho^{d+1}+
\frac{2pc_d}{\alpha}2^p\left| J_{\alpha}\right|,
% &\quad  +\frac{2pc_d}{\alpha}\left|\int_{B_R^c(n(x))\cap B_\rho^c(0)}
% \eta_R(n(x),y)
%  V(\alpha |x| y) \cdot \frac{\partial}{\partial y_j}
% \frac{y}{(|y|^2+|x|^{-2})^{p+1}}\, dy\right|\\
% & \quad +\frac{2pc_d}{\alpha}\left| J_{\alpha}\right|,
\end{aligned}
\end{equation}
where
\begin{align*}
 J_{\alpha}& :=\int_{B_\rho^c(0)}
f(n(x),y)
 V(\alpha |x| y) \cdot \frac{\partial}{\partial y_j}
\frac{y}{(|y|^2+|x|^{-2})^{p+1}}\, dy.
%\\
%f(x,y)&:=\frac{x-y}{|x-y|^d} - \frac{x-y}{|x-y|^{d-1}}.
\end{align*}
To estimate $J_\alpha$, we write
\begin{equation*}\label{eq_Ja_0}
\begin{aligned}
 J_{\alpha} & =\int_{B_\rho^c(0)}
f(n(x),y)
 V(\alpha |x| y) \cdot \frac{\partial}{\partial y_j}
\frac{y}{(|y|^2+|x|^{-2})^{p+1}}\, dy\\
&=\int_{B_\rho^c(0)\cap B_R^c(n(x))}
\eta_R(n(x),y)
 V(\alpha |x| y) \cdot \frac{\partial}{\partial y_j}
\frac{y}{(|y|^2+|x|^{-2})^{p+1}}\, dy \\
&\quad +
\int_{B_R(n(x))}
f(n(x),y)
 V(\alpha |x| y) \cdot \frac{\partial}{\partial y_j}
\frac{y}{(|y|^2+|x|^{-2})^{p+1}}\, dy,
\end{aligned}
\end{equation*}
and calculating
\[
\frac{\partial}{\partial y_j}
\frac{y}{(|y|^2+|x|^{-2})^{p+1}}\, dy = \frac{e_j}{(|y|^2+|x|^{-2})^{p+1}} - 2(p+1)\frac{y_jy}{(|y|^2+|x|^{-2})^{p+2}},
\]
we obtain
\begin{equation}\label{eq_Ja}
\begin{aligned}
 J_{\alpha} & =
\int_\rho^{+\infty} \frac{r^{d-1}}{(r^2+|x|^{-2})^{p+1}} \, dr
\int_{\partial B_1(0)} \eta_R(n(x),rs) V_j (\alpha rs)
\,d\HH^{d-1}(s)\\
&\quad - 2(p+1) \int_\rho^{+\infty} \frac{r^{d+1}}{(r^2+|x|^{-2})^{p+2}} \, dr
\\
&\qquad\qquad
\int_{\partial B_1(0)} \eta_R(n(x),rs) s_j
s\cdot V(\alpha r s)
\,d\HH^{d-1}(s)\\
&\quad + \int_{B_R(n(x))}
f(n(x),y)
 V(\alpha |x| y) \cdot \frac{\partial}{\partial y_j}
\frac{y}{(|y|^2+|x|^{-2})^{p+1}}\, dy,
\end{aligned}
\end{equation}
where $R\in (0, 1-\rho)$ will be chosen later, and
 \begin{equation*}\label{eq_defphi2}
 \eta_R(x,y):= \left\{
 \begin{array}{rl}
  \frac{x-y}{|x-y|^d} - \frac{x-y}{|x-y|^{d-1}}, & y\in B_R^c(x),\\
 \frac{x-y}{R^d} - \frac{x-y}{R^{d-1}}
 , & y\in B_R(x).
 \end{array}
 \right.
 \end{equation*}
Since for $y\in B_R(n(x))$ one has
\begin{align*}
\left|\frac{\partial}{\partial y_j}
\frac{y_i}{(|y|^2+|x|^{-2})^{p+1}}\right| &=
\frac{|\delta_{ij} - 2(p+1)y_iy_j|}{(|y|^2+|x|^{-2})^{p+2}}\\
&
\leq
\frac{1 + 2(p+1)R^2}{|y|^{2(p+2)}} %\quad\mbox{because $|y|<R$}\\
%&
\leq  C_R:=\frac{1 + 2(p+1)R^2}{|1-R|^{2(p+2)}},
\end{align*}
choosing $R$ so that
\[
\int_{B_R(n(x))} |f(n(x),y)|\, dy <\delta
\]
for some $\delta>0$, we get
\begin{equation}\label{eq_Ja1}
\left|\int_{B_R(n(x))}
f(n(x),y)
 V(\alpha |x| y) \cdot \frac{\partial}{\partial y_j}
\frac{y}{(|y|^2+|x|^{-2})^{p+1}}\, dy\right| \leq \|V\| C_R \delta.
\end{equation}
But
\begin{eqnarray*}
\left|\int_\rho^{+\infty} \right. &\left.\displaystyle \frac{r^{d-1}}{(r^2+|x|^{-2})^{p+1}} \, dr
\int_{\partial B_1(0)} \eta_R(n(x),rs) V_j (\alpha rs)
\,d\HH^{d-1}(s)\right|\\
&\leq\displaystyle \|V\|_\infty \|\eta_R\|_\infty d\omega_d\int_\rho^{+\infty}  \frac{r^{d-1}}{r^{2(p+1)}}\,dr,\\
\left| \int_\rho^{+\infty}\right. &\left. \displaystyle\frac{r^{d+1}}{(r^2+|x|^{-2})^{p+2}} \, dr
\int_{\partial B_1(0)} \eta_R(n(x),rs) s_j
s\cdot V(\alpha r s)  \,d\HH^{d-1}(s)\right| \\
 &\leq \|V\|_\infty \|\eta_R\|_\infty d\omega_d\displaystyle\int_\rho^{+\infty}\frac{r^{d+1}}{r^{2(p+2)}} \, dr,
\end{eqnarray*}
the integrals on the right-hand side of the above inequalities being convergent
(because $p> d/2-1$), so that combining this with~\eqref{eq_Ja1}
we get from~\eqref{eq_Ja} that
$|J_{\alpha}|$ is bounded by a constant independent on $\alpha$ (hence
depending only on $p$, $d$, $\|V\|_\infty$). Therefore, from~\eqref{eq_estL1a}
we have that $|\Lambda_1(x)|\leq \varepsilon /6$ for
$\alpha$ sufficiently large (depending on $p$ and $\varepsilon$).

Finally, it remains to estimate $\Lambda_2$. To this
aim
%consider
%for every $i\in \{1, \ldots, d\}$ the component
 integrating by parts the expression for $\Lambda_2$, we get
 (with notation $V^T$ standing for $V$ transposed, i.e. seen as a row, while
 by default the vectors are seen as columns)
\[
\begin{aligned}
\Lambda_2 (x)&=
% \frac{2pc_d}{\alpha|x|} (1+ |x|^{-2})^p \int_{\R^d}
% \frac{n_i(x)-y}{|n(x)-y|^{d-1}}
%  V(\alpha |x| y) \cdot \frac{\partial}{\partial y_j}
% \frac{y}{(|y|^2+|x|^{-2})^{p+1}}\, dy\\
% &=
-\frac{2pc_d}{\alpha|x|} (1+ |x|^{-2})^p\int_{\R^d}
\frac{\partial}{\partial y_j}\left(\frac{n(x)-y}{|n(x)-y|^{d-1}}
 V^T(\alpha |x| y)\right)
\frac{y\,dy}{(|y|^2+|x|^{-2})^{p+1}}\\
&=
-\frac{2pc_d}{\alpha|x|} (1+ |x|^{-2})^p\int_{\R^d}
\frac{\partial}{\partial y_j}\left(\frac{n(x)-y}{|n(x)-y|^{d-1}}\right)
 V(\alpha |x| y) \cdot
\frac{y\,dy}{(|y|^2+|x|^{-2})^{p+1}}\\
&\quad -2pc_d(1+ |x|^{-2})^p\int_{\R^d}\frac{n(x)-y}{|n(x)-y|^{d-1}}
 V_{y_j}(\alpha |x| y) \cdot
\frac{y\,dy}{(|y|^2+|x|^{-2})^{p+1}}\\
&=
\frac{2pc_d}{\alpha|x|} (1+ |x|^{-2})^p\int_{\R^d}\frac{e_j}{|n(x)-y|^{d-1}}
 V(\alpha |x| y) \cdot
\frac{y\,dy}{(|y|^2+|x|^{-2})^{p+1}} \\
&\quad
-\frac{2p(d-1)c_d}{\alpha|x|} (1+ |x|^{-2})^p\\
&\qquad\quad\int_{\R^d}
\left(\frac{n(x)-y}{|n(x)-y|^{d+1}}(n_j(x)-y_j)\right)
 V(\alpha |x| y) \cdot
\frac{y\,dy}{(|y|^2+|x|^{-2})^{p+1}}\\
&\quad -2pc_d(1+ |x|^{-2})^p\int_{\R^d}\frac{n(x)-y}{|n(x)-y|^{d-1}}
 V_{y_j}(\alpha |x| y) \cdot
\frac{y\,dy}{(|y|^2+|x|^{-2})^{p+1}},
\end{aligned}
\]
% the latter expression being obtained by integration by parts. Thus
% \[
% \begin{aligned}
% \Lambda_2^i (x)&=
% -\frac{2pc_d}{\alpha|x|} (1+ |x|^{-2})^p\int_{\R^d}
% \frac{\partial}{\partial y_j}\left(\frac{n_i(x)-y}{|n(x)-y|^{d-1}}\right)
%  V(\alpha |x| y) \cdot
% \frac{y\,dy}{(|y|^2+|x|^{-2})^{p+1}}\\
% &\quad -\int_{\R^d}\frac{n_i(x)-y}{|n(x)-y|^{d-1}}
% \alpha |x| V_{y_j}(\alpha |x| y) \cdot
% \frac{y\,dy}{(|y|^2+|x|^{-2})^{p+1}} ,
% \end{aligned}
% \]
and the desired estimate for $\Lambda_2$ follows from Lemma~\ref{lm_NPsmall2}
applied to each of the three integrals in the right-hand side of the above equality
(for the second integral it has to be applied with $A(x,y):=(x_i-y_i)(x_j-y_j)/|x-y|^2$, $k:=d-1$, for all
$i=1,\ldots, d$).
\end{proof}

\section{Auxiliary lemmata}

In this section we collect some easy auxiliary results of ``folkloric'' nature.

\begin{lemma}\label{lm_invmeas_div1}
Suppose that $V\colon \R^d\to \R^d$ is a bounded continuous vector field providing the uniqueness of
the solution to the Cauchy problem for~\eqref{eq_odeaut1} with any initial datum, hence generating the flow
$T_t:=\varphi^t_V\colon \R^d\to \R^d$, $t\in \R$, of the latter equation.
If $\mu$ is a $\sigma$-finite Radon measure over $\R^d$
satisfying
$\div (V\mu)=0$ in the weak
(i.e.\ distributional) sense in $\R^d$, then it is invariant with respect to $T_t$.
%The Radon measure $\mu$ over $\R^d$ is invariant with respect to $T_t$, if and only if $\div (V\mu)=0$ in the weak
%(i.e.\ distributional) sense in $\R^d$.
\end{lemma}

This lemma is well-known in a very particular case when $V$ is smooth and
$\mu$ is the Lebesgue measure, though for the general situation its proof maybe not quite immediate.

\begin{proof}
%Let $f\in C^1(\R^d)$ have compact support. Suppose first that $\mu$ is invariant with respect to $T_t$. One has then
%\begin{equation}\label{eq_invmeas_div1}
%\begin{aligned}
%0& =\frac{d\,}{dt} \int_{\R^d} f\, d\mu= \frac{d\,}{dt} \int_{\R^d} f\, d{T_{t\#}}\mu=
% \frac{d\,}{dt} \int_{\R^d} f(T_t(x))\, d\mu(x)\\
% & =\int_{\R^d} \frac{d\,}{dt}  f(T_t(x))\, d\mu(x),
%\end{aligned}
%\end{equation}
%the latter equality being due to Lebesgue dominated convergence theorem since
%\begin{equation}\label{eq_invmeas_div2}
%\frac{d\,}{dt}  f(T_t(x))= \nabla f(T_t(x))\cdot V(T_t(x))
%\end{equation}
%is bounded and has compact support, hence is integrable with respect to $\mu$.
%Plugging~\eqref{eq_invmeas_div2} into~\eqref{eq_invmeas_div1}, we get
%\[
%0=\int_{\R^d} \nabla f(T_t(x))\cdot V(T_t(x))\, d\mu(x)=\int_{\R^d} \nabla f(x)\cdot V(x)\, d\mu(x),
%\]
%the latter again by invariance of $\mu$ with respect to $T_t$, proving that $\div (V\mu)=0$ in the weak $\div (V\mu)=0$ in the weak
%(i.e.\ distributional) sense in $\R^d$.
%
%Vice versa, suppose $\div (V\mu)=0$ in the weak sense, that is,
%%\[
%%\begin{equation}\label{eq_invmeas_div3}
%%\int_{\R^d} \nabla f\cdot V\, d\mu=0
%%\end{equation}
%%\]
%%for every smooth $f$ with compact support in $\R^d$. In other words,
The ``constant'' curve of measures $\{\mu_t\}_{t\in \R^+}$ defined by $\mu_t:=\mu$ for all $t\in \R^+$,
clearly satisfies the continuity equation
\begin{equation}\label{eq_invmeas_div3}
\partial_t \mu_t + \div(V\mu_t)=0
\end{equation}
in the weak (distributional) sense in $\R^+\times\R^d$.
Consider an arbitrary $\varphi\in C^1(\R^d)$ and set $\varphi_t(x):= \varphi(T_t^{-1}(x))$ (note that unders the conditions of the Lemma being proven
$T_t$ is in fact a homeomorphism).
We show
\begin{equation}\label{eq_invmeas_div3a}
\partial_t (\varphi_t \mu) + \div(V\,\varphi_t \mu)=0
\end{equation}
in the weak sense in $\R^+\times\R^d$.
In fact,
\begin{align*}
0& =\frac{d\,}{dt} \varphi(x)= \frac{d\,}{dt} \varphi_t(T_t(x)) = (\partial_t\varphi_t)(T_t(x)) +(\nabla\varphi_t)(T_t(x))\cdot V(T_t(x))
\end{align*}
for all $x\in \R^d$ in view of~\eqref{eq_odeaut1}, and hence
\begin{equation}\label{eq_invmeas_div3b}
\partial_t\varphi_t(x) +\nabla\varphi_t(x)\cdot V(x)=0
\end{equation}
for all $x\in \R^d$. Therefore, using the chain rule for distributional derivatives,~\eqref{eq_invmeas_div3b} and~\eqref{eq_invmeas_div3}, we get
\begin{align*}
\partial_t (\varphi_t \mu) + \div(V\,\varphi_t \mu) & = \mu \left(\partial_t \varphi_t  + \nabla\varphi_t \cdot V\right) +
\varphi_t\left(\partial_t \mu + \div(V\mu)\right)=0,
\end{align*}
that is,~\eqref{eq_invmeas_div3a}.

On the other hand, clearly, $T_{t\#}(\varphi\mu):= \varphi_t T_{t\#}\mu$, and hence
\begin{equation}\label{eq_invmeas_div4}
\partial_t (\varphi_t T_{t\#}\mu) + \div(V\,\varphi_t T_{t\#}\mu)=0
\end{equation}
in the weak sense in $\R^+\times\R^d$ because of~\eqref{eq_odeaut1}. Assuming now that $\varphi$ is nonnegative and have compact support in $\R^d$, we have that both curves of finite positive Borel measures $\{\varphi_t T_{t\#}\mu\}_{t\in \R^+}$ and
$\{\varphi_t \mu\}_{t\in \R^+}$ satisfy the continuity equation, and have the same initial point $\varphi\mu$.
By the superposition principle for curves of measures
(see theorem~12 in~\cite{ambrosio-crippa-edinburgh} for its statement in the Euclidean space, or theorem~3.21 in~\cite{SteTrev14} for the general metric space version)
for each fixed $T>0$ there are
finite positive Borel measures $\eta$ and $\tilde\eta$ on $C([0,T]; \R^d)$ each concentrated over trajectories of~\eqref{eq_odeaut1} such that $e_{t\#}\eta=\varphi_t T_{t\#}\mu$ and $e_{t\#}\tilde\eta=\varphi_t \mu$ for a.e.\
$t\in [0,T]$, $e_{0\#}\eta=e_{0\#}\tilde\eta=\varphi \mu$, where $e_t\colon C([0,T]; \R^n)\to \R^d$ stands for the evaluation map
$e_t(\theta):=\theta(t)$. Disintegrating
\[
\eta=(\varphi\mu)\otimes \eta_x, \qquad \tilde \eta=(\varphi\mu)\otimes \tilde\eta_x
\]
with $\eta_x$ and $\tilde \eta_x$ Borel probability measures concentrated each over the trajectory of~\eqref{eq_odeaut1}
starting at $x\in \R^d$, we have that $\eta_x=\tilde\eta_x$ for $\varphi\mu$-a.e.\ $x\in \R^d$ by the assumption on unique solvability of~\eqref{eq_odeaut1}. Thus
$\eta=\tilde \eta$, which implies $\varphi_t T_{t\#}\mu=\varphi_t \mu$ for a.e. $t\in [0,T]$, and hence, since $\varphi$ is
arbitrary (nonnegative with compact support) and $T>0$ is arbitrary, we get that $T_{t\#}\mu=\mu$ for a.e., and thus
for all $t\in \R^+$, i.e.\
$\mu$ is invariant as claimed.
\end{proof}

\begin{remark}
In fact, under the conditions of the above Lemma~\ref{lm_invmeas_div1} the Radon measure $\mu$ over $\R^d$ is invariant with respect to $T_t$, if and only if $\div (V\mu)=0$ in the weak
sense in $\R^d$. The ``if'' part is however trivial: for $f\in C^1(\R^d)$ with compact support we get
\begin{equation}\label{eq_invmeas_div1}
\begin{aligned}
0& =\frac{d\,}{dt} \int_{\R^d} f\, d\mu= \frac{d\,}{dt} \int_{\R^d} f\, d{T_{t\#}}\mu=
 \frac{d\,}{dt} \int_{\R^d} f(T_t(x))\, d\mu(x)\\
 & =\int_{\R^d} \frac{d\,}{dt}  f(T_t(x))\, d\mu(x),
\end{aligned}
\end{equation}
the latter equality being due to the Lebesgue dominated convergence theorem since
\begin{equation}\label{eq_invmeas_div2}
\frac{d\,}{dt}  f(T_t(x))= \nabla f(T_t(x))\cdot V(T_t(x))
\end{equation}
is bounded and has compact support, hence is integrable with respect to $\mu$.
Plugging~\eqref{eq_invmeas_div2} into~\eqref{eq_invmeas_div1}, we get
\[
0=\int_{\R^d} \nabla f(T_t(x))\cdot V(T_t(x))\, d\mu(x)=\int_{\R^d} \nabla f(x)\cdot V(x)\, d\mu(x),
\]
the latter again by invariance of $\mu$ with respect to $T_t$, proving that $\div (V\mu)=0$ weakly in $\R^d$.
\end{remark}

It is worth emphasizing that the above Lemma~\ref{lm_invmeas_div1} can be stated and proven in a similar way even for less regular vector fields $V$. However, in the present paper we apply it only in a very particular situation when $V$ is bounded and locally Lipschitz, and $\mu=\psi\,dx$ with $\psi\colon \R^d\to \R^+$ smooth. In this case its proof %of the ``only if'' (i.e.\ second) part of
%the above Lemma
could be clearly simplified without even referring to general superposition principle, but just by observing ``manually'' the uniqueness of positive solutions to continuity equation. We provided here a more general (and less common) version just for the readers' convenience.

Another easy statement we use in the paper is as follows.

\begin{lemma}\label{lm_invsupp1}
If $T\colon X\to X$ is continuous, has an
invariant measure $\mu$, and $T(\supp\mu)$ is closed (in particular, this
is true if $\mu$ has compact support).
Then
$\supp \mu$ is %positively
invariant for $T$.
\end{lemma}

\begin{proof}
Denoting for brevity $M:=\supp \mu$, we have that $T^{-1}(M)$
is closed and $\mu((T^{-1} (M))^c)=\mu(T^{-1} (M^c))=0$, in other words, $\mu$
is concentrated on $T^{-1}(M)$,
which implies
$M\subset T^{-1}(M)$, thus $T(M)\subset M$.
On the other hand,
\begin{align*}
\mu((T(M))^c) &= (T_{\#}\mu)((T(M))^c)\\
& = \mu(T^{-1}((T(M))^c)=
\mu((T^{-1}(T(M))^c)\leq \mu(M^c)=0,
\end{align*}
the inequaity in the above chain being due to
$M\subset T^{-1}(T(M))$. Thus
$\mu$
is also concentrated on $T(M)$
and since $T(M)$ is also closed, we have $M\subset T(M)$ concluding the proof.
% One has to show that
% \[
% M^+:= \{x\in \supp\mu\colon T(x)\not\in \supp\mu\}, \quad
% M^-:= \{x\in \supp\mu\colon T^{-1}(x)\cap \supp\mu=\emptyset\}
% \]
% are emptysets. Observe that they
% are $\mu$-negligible and in particular dense in $M$, because
% \begin{align*}
% \mu(M^+)& \leq \mu(T^{-1}(T(M^+)))= \mu(T(M^+))=0,\\
% \mu(M^-)& = \mu(T^{-1}(M^-))=0.
% \end{align*}
% Now, $M^+=\emptyset$, since otherwise there is an $x\in M^+$
% and a sequence $x_k\in M\setminus M^+$ such that $\lim_k x_k =x$ (in fact,
% $x$ cannot be isolated, because $\mu(M^+)=0$), so that
% $T(x)=\lim_k T(x_k)\in \supp\mu$
% (because all $T(x_k)\in \supp\mu$) contradicting the assumption on $x$.
% Analogously, $M^-=\emptyset$, since otherwise there is an $x\in M^-$
% (which again cannot be isolated in view of $\mu(M^+)=0$) and
% a sequence $x_k\in M\setminus M^-$ such that $\lim_k x_k =x$, but
% $x_k\in T(\supp\mu)$, the latter set being compact, so that
% $x\in T(\supp\mu)$ contradicting the assumption on $x$.
\end{proof}

Note that continuity of $T$ in the above Lemma~\ref{lm_invsupp1} is essential.

\bibliographystyle{plain}
%\bibliography{invmeas}

\end{document}